\documentclass[a4paper,12pt]{amsart}
\usepackage{amssymb,amscd,amsmath}
\usepackage[dvips]{graphicx}

\newcounter{mytheorem} 

\newtheorem{theorem}[mytheorem]{Theorem}
\newtheorem{definition}[mytheorem]{Definition}
\newtheorem{lemma}[mytheorem]{Lemma}
\newtheorem{claim}[mytheorem]{Claim}
\newtheorem{proposition}[mytheorem]{Proposition}

\newtheorem{problem*}{Problem}
\newtheorem{example}[mytheorem]{Example}

\newtheorem{remark}[mytheorem]{Remark}

\newtheorem{corollary}[mytheorem]{Corollary}
\newtheorem{clm}[mytheorem]{Claim}
\newtheorem*{clm*}{Claim}

\makeatletter
\@addtoreset{subsection}{section}
\@addtoreset{subsubsection}{section}
\@addtoreset{subsubsection}{subsection}
\@addtoreset{equation}{section}
\makeatother

\makeatletter
\newcommand\testshape{family=\f@family; series=\f@series; shape=\f@shape.}
\def\myemphInternal#1{\if n\f@shape%
\begingroup\itshape #1\endgroup\/%
\else\begingroup\bfseries #1\endgroup%
\fi}
\def\myemph{\futurelet\testchar\MaybeOptArgmyemph}
\def\MaybeOptArgmyemph{\ifx[\testchar \let\next\OptArgmyemph
                 \else \let\next\NoOptArgmyemph \fi \next}
\def\OptArgmyemph[#1]#2{\index{#1}\myemphInternal{#2}}
\def\NoOptArgmyemph#1{\myemphInternal{#1}}
\makeatother

\newcommand\RRR{\mathbb{R}}
\newcommand\QQQ{\mathbb{Q}}
\newcommand\CCC{\mathbb{C}}
\newcommand\ZZZ{\mathbb{Z}}
\newcommand\NNN{\mathbb{N}}

\newcommand\Int[1]{\mathrm{Int}(#1)}
\newcommand\cl[1]{\overline{#1}}

\newcommand\Per{\mathrm{Per}}

\newcommand\id{\mathrm{id}}

\newcommand\diam{\mathrm{diam}}

\newcommand\AFld{F}
\newcommand\AFlow{\mathbf{\AFld}}
\newcommand\BFld{G}
\newcommand\BFlow{\mathbf{\BFld}}

\newcommand\Mman{M}

\newcommand\pfunc{\xi}
\newcommand\nufunc{\nu}

\newcommand\tmufunc{\widetilde{\pfunc}}

\newcommand\Bman{B}

\newcommand\Qman{Q}

\newcommand\Vman{V}
\newcommand\Uman{U}
\newcommand\Wman{W}

\newcommand\eps{\varepsilon}

\newcommand\Cnt[1]{\mathcal{C}}
\newcommand\Cont[1]{\mathcal{C}^{#1}}

\newcommand\Cr[3]{\Cont{#1}(#2,#3)}

\newcommand\Cz[2]{\Cr{}{#1}{#2}}

\newcommand\orb{o}

\newcommand\FixF{\Sigma}
\newcommand\FixA{\FixF_{\AFld}}

\newcommand\dif{h}

\newcommand\orig{0} 

\newcommand\PF{P}
\newcommand\RPF[1]{RP(#1)}
\newcommand\PPF[1]{P(#1)}

\newcommand\ThetaPropPositive{{\rm(1)}}
\newcommand\ThetaPropPeriod{{\rm(2)}}
\newcommand\ThetaPropRegularOpen{{\rm(3)}}
\newcommand\ThetaPropRegularRegCl{{\rm(4)}}
\newcommand\ThetaPropCircleAction{{\rm(4)}}

\newcommand\upb{\mathsf{C}}
\newcommand\lpb{\mathsf{L}}
\newcommand\MULA{\mathsf{A}}
\newcommand\MULT{\mathsf{T}}
\newcommand\MULS{\mathsf{S}}
\newcommand\Jord{\mathbf{J}}

\begin{document}	

\title{Period functions for $\Cont{0}$ and $\Cont{1}$ flows}
\author{Sergiy Maksymenko}
\address{Topology dept., Institute of Mathematics of NAS of Ukraine, Tere\-shchenkivs'ka st. 3, Kyiv, 01601 Ukraine}
\date{2009/10/15}
\email{maks@imath.kiev.ua}

\begin{abstract}
Let $\mathbf{F}:M\times\mathbb{R}\to M$ be a continuous flow on a topological manifold $M$.
For every open $V\subset M$ denote by $P(V)$ the set of all continuous functions $\pfunc:V\to\mathbb{R}$ such that $\mathbf{F}(x,\pfunc(x))=x$ for all $x\in V$.
Such functions vanish at non-periodic points and their values at periodic points are equal to the corresponding periods (in general not minimal).
In this paper $P(V)$ is described for connected open subsets $\Vman\subset\Mman$. 

It is also shown that there is a difference between the sets $P(V)$ for $C^{0}$ and $C^1$ flows: if $V$ is open and connected, and the flow $\mathbf{F}$ is $C^1$, then any $\pfunc\in P(V)$ being not identically zero on $V$ is everywhere non-zero on $V$, while for $C^0$ flows there can exist functions $\pfunc\in P(V)$ vanishing at fixed points.
\end{abstract}

\maketitle



\section{Introduction}
Let $\AFlow:\Mman\times\RRR\to\Mman$ be a continuous flow on a topological finite-dimensional manifold $\Mman$.
Let also $\FixF$ be the set of fixed points of $\AFlow$.
For $x\in\Mman$ we will denote by $\orb_x$ the orbit of $x$.
If $x$ is periodic, then $\Per(x)$ is the period of $x$.

\begin{definition}\label{defn:P-func}
Let $\Vman \subset \Mman$ be a subset and $\pfunc:\Vman\to\RRR$ be a continuous function.
We will say that $\pfunc$ is a \myemph{period} function or simply a \myemph{$\PF$-function} (with respect to $\AFlow$) if $\AFlow(x,\pfunc(x))=x$ for all $x\in\Vman$.
The set of all $\PF$-functions on $\Vman$ will be denoted by $\PPF{\Vman}$.
\end{definition}

The aim of this paper is to give a description of $\PPF{\Vman}$ for open connected subsets $\Vman\subset\Mman$ with respect to a  continuous flow on a topological manifold $\Mman$ (Theorem\;\ref{th:description_PF_V}).
Such a description was given in \cite[Th.\;12]{Maks:TA:2003} for $\Cont{\infty}$ flows.
Unfortunately, the proof of a key statement \cite[Pr.\;10]{Maks:TA:2003} for \cite[Th.\;12]{Maks:TA:2003} contains a gap.
We will show that in fact that proposition is true (Proposition~\ref{pr:mu_in_kerA_vanish}).
Moreover, we also indicate the difference between $\PF$-functions for $\Cont{0}$ and $\Cont{1}$ flows (Theorem\;\ref{th:C1-flows-non-periodic}).
Our methods are based on well-known theorems of M.\;Newman about actions of finite groups.

\smallskip

The following easy lemma explains the term \myemph{$\PF$-function}.
The proof is the same as in \cite[Lm.\;5\,\&\,7]{Maks:TA:2003} and we leave it for the reader.
\begin{lemma}\label{lm:init_prop_ShAV}{\rm\;\cite[Lm.\;5\,\&\,7]{Maks:TA:2003}}
For any subset $\Vman\subset\Mman$ the set $\PPF{\Vman}$ is a group with respect to the point-wise addition.

Let $x\in\Vman$ and $\pfunc\in\PPF{\Vman}$.
Then $\pfunc$ is locally constant on $\orb_x\cap\Vman$.
In particular, if $x$ is non-periodic, then $\pfunc|_{\orb_{x}\cap\Vman}=0$.
Suppose $x$ is periodic, and let $\omega$ be some path component of $\orb_{x}\cap\Vman$.
Then $\pfunc(y)=n\cdot\Per(x)$ for some $n\in\ZZZ$ depending on $\omega$ and all $y\in\omega$.
\end{lemma}

It is not true that any $\PF$-function on any subset $\Vman\subset\Mman$ is constant on all of $\orb_{x}\cap\Vman$ for each $x\in\Vman$, see Example\;\ref{exmp:nonreg-P-function} below.
Therefore we give the following definition.

\begin{definition}\label{defn:regular_P_func}
A $\PF$-function $\pfunc:\Vman\to\RRR$ will be called \myemph{regular} if $\pfunc$ is constant on $\orb_x\cap\Vman$ for each $x\in\Vman$.
\end{definition}

Denote by $\RPF{\Vman}$ the set of all regular $\PF$-functions on $\Vman$.
Then $\RPF{\Vman}$ is a subgroup of $\PPF{\Vman}$.

\begin{remark}\label{rem:o_cap_V_connected_mu_is_regular}\rm
If for any periodic orbit $\orb$ the intersection $\orb\cap\Vman$ is either empty or connected, e.g. in the case when $\Vman$ is $\AFlow$-invariant, then any $\PF$-function on $\Vman$ is regular.
\end{remark}

The following theorem extends\;\cite[Th.\;12]{Maks:TA:2003} for continuous case%
\footnote{Unfortunately, the proof of\;\cite[Pr.\;10]{Maks:TA:2003} being the crucial step for\;\cite[Th.\;12]{Maks:TA:2003} is incorrect.
It was wrongly claimed that (using notations from \cite{Maks:TA:2003}) \myemph{if $\Phi_t$, $t\in(-\eps,\eps)$, is a local flow on $\RRR^n$, and $x\in\RRR^n$, then the map $\Psi(x,*):(-\eps,\eps)\to GL_{n}(\RRR)$, associating to each $t$ the Jacobi matrix of the flow map $\Phi_t$ at $x$, is a homomorphism}.
This is not so since $\Psi$ is determined by a particular choice of the trivialization of tangent bundle along the orbit of $x$ and.
In fact, $\Psi$ cane be any smooth map such that $\Psi(x,0)$ is a unit matrix.

Nevertheless, the statement of\;\cite[Pr.\;10]{Maks:TA:2003} and the arguments which deduce from it \cite[Th.\;12]{Maks:TA:2003} are true.
We will extend\;\cite[Pr.\;10]{Maks:TA:2003} in Proposition\;\ref{pr:mu_in_kerA_vanish} and thus present a correct proof.
\label{rem:incorrect_proof_Pr10_Maks:Shifts}
}.

\begin{theorem}\label{th:description_PF_V}
Let $\Mman$ be a finite-dimensional topological manifold possibly non-compact and with or without boundary, $\AFlow:\Mman\times\RRR\to\Mman$ be a flow, and $\Vman\subset\Mman$ be an \myemph{open, connected set}.

{\rm(A)}
If $\Int{\FixF}\cap\Vman\not=\varnothing$, then 
$$
\PPF{\Vman}=\{\pfunc\in\Cz{\Vman}{\RRR} \ : \ \pfunc|_{\Vman\setminus\FixF}=0\}.
$$

{\rm(B)}
Suppose $\Int{\FixF}\cap\Vman=\varnothing$.
Then one of the following possibilities is realized: either
$$\PPF{\Vman}=\{0\}$$
or
$$\PPF{\Vman}=\{n\theta\}_{n\in\ZZZ}$$
for some continuous function $\theta:\Vman\to\RRR$ having the following properties:
\begin{enumerate}
 \item[\ThetaPropPositive]
$\theta>0$ on $\Vman\setminus\FixF$, so this set consists of periodic points only.
 \item[\ThetaPropPeriod]
There exists an open and everywhere dense subset $\Qman\subset\Vman$ such that $\theta(x)=\Per(x)$ for all $x\in\Qman$.
 \item[\ThetaPropRegularOpen]
$\theta$ is regular.
 \item[\ThetaPropCircleAction]
Denote $\Uman=\AFlow(\Vman\times\RRR)$.
Then $\theta$ extends to a $\PF$-function on $\Uman$ and there is a circle action $\BFlow:\Uman\times S^1\to\Uman$ defined by $\BFlow(x,t) =\AFlow(x,t\theta(x))$, $x\in\Uman$, $t\in S^1 =\RRR/\ZZZ$.
The orbits of this action coincides with the ones of $\AFlow$.
\end{enumerate}
In particular, in all the cases $\RPF{\Vman}=\PPF{\Vman}$.
\end{theorem}
Theorem\;\ref{th:description_PF_V} will be proved in \S\ref{sect:proof:th:description_PF_V}.

\subsection{$\PF$-functions for $\Cont{1}$ flows.}
Let $\dif:\Mman\to\Mman$ be a homeomorphism and $\BFlow:\Mman\times\RRR\to\Mman$ be the conjugated flow: 
$$
\BFlow_t(x) = \dif\circ\AFlow_t\circ\dif^{-1}(x) = \dif\circ\AFlow(\dif^{-1}(x),t)).
$$
If $\Vman\subset\Mman$ is an open set and $\theta:\Vman\to\RRR$ a continuous $\PF$-function for $\AFlow$, then $\theta\circ\dif:\dif^{-1}(\Vman) \to \RRR$ is a $\PF$-function for $\BFlow$.
Indeed, let $x\in\Vman$ and $y=\dif^{-1}(x)$.
Then 
$$
\BFlow(y,\theta\circ\dif^{-1}(y)) =
\dif\circ\AFlow(\dif^{-1}(x),\theta\circ\dif^{-1}(x))) =
\dif\circ\dif^{-1}(x)=x.
$$

It follows that the structure of the set of $\PF$-functions of the flow $\AFlow$ depends only on its conjugacy class.

\smallskip 

Now let $\Mman$ be $\Cont{r}$, $(r\geq1)$, manifold and $\AFlow:\Mman\times\RRR\to\Mman$ be a $\Cont{r}$ flow.
Then in general, $\AFlow$ is generated by a $\Cont{r-1}$ vector field
$$
\AFld(x) = \tfrac{\partial\AFlow}{\partial t}(x,t)|_{t=0}.
$$
Nevertheless, it is proved by D.\;Hart\;\cite{Hart:Top:1983} that every $\Cont{r}$ flow $\AFlow$ is $\Cont{r}$-conjugated to a $\Cont{r}$ flow generated by a $\Cont{r}$ vector field.
Thus in order to study $\PF$-functions for $\Cont{r}$ flows we can assume that these flows are generated by $\Cont{r}$ vector fields.

\begin{theorem}\label{th:C1-flows-non-periodic}
Let $\AFlow$ be a $\Cont{1}$ flow on a connected manifold $\Mman$ generated by a $\Cont{1}$ vector field $\AFld$.
Suppose $\PPF{\Vman}=\{n\theta\}_{n\in\ZZZ}$ for some non-negative $\PF$-function $\theta:\Vman\to[0,+\infty)$, see {\rm(1)} of Theorem\;\ref{th:description_PF_V}. 
Then, in fact, $\theta>0$ on all of $\Mman$.

Moreover, for every $z\in\FixF$ there are local coordinates in which the linear part $j^1\AFld(z)$ of $\AFld$ at $z$ is given by the following matrix
\begin{equation}\label{equ:j1Fz_ShA_periodic}
\left(\begin{smallmatrix} 
0 & \beta_1 \\ -\beta_1 & 0  \\
 &  & \cdots \\
 &  &    & 0 & \beta_k \\
 &  &    & -\beta_k & 0 \\
 & & & & & 0 \\
 & & & & & & \cdots 
\end{smallmatrix}\right),
\end{equation}
for some $k\geq1$ and $\beta_j\in\RRR\setminus0$.
\end{theorem}

We will prove this theorem in \S\ref{sect:proof:th:C1-flows-non-periodic}.

\begin{remark}\label{rem:proof:th:C1-flows-non-periodic:theta_C1}\rm
Under assumptions of Theorem\;\ref{th:C1-flows-non-periodic} suppose that $\theta$ is $\Cont{1}$ and $\theta(z)\not=0$ for some $z\in\FixF$.
In this case existence of\;\eqref{equ:j1Fz_ShA_periodic} at $z$ is easy to prove, c.f.\;\cite{Maks:reparam-sh-map}.

Indeed, define the flow $\BFlow:\Mman\times\RRR\to\Mman$ by $\BFlow(x,t)=\AFlow(x,t\,\theta(x))$.
Then $\BFlow$ is generated by the $\Cont{1}$ vector field $\BFld=\theta\AFld$.
Moreover, $\BFlow_{1}=\id_{\Mman}$, whence $\BFlow$ yields an $\RRR/\ZZZ=S^1$-action on $\Mman$.

Let $z\in\FixF$.
Then $\BFlow_t(z)=z$, whence $\BFlow$ yields a linear $S^1$-action $T_z\BFlow_t$ on the tangent space $T_z\Mman$.
Now it follows from standard results about $S^1$ representations in $GL(\RRR,n)$ that the linear part of $\BFld$ at $z$ in some local coordinates is given by\;\eqref{equ:j1Fz_ShA_periodic}.
It remains to note that $j^1\AFld(z) = j^1\BFld(z)/\theta(z)$.
Notice that these arguments \myemph{do not prove that the matrix\;\eqref{equ:j1Fz_ShA_periodic} is non-zero.}
\end{remark}

\subsection{Structure of the paper}
In next section we will consider examples illustrating necessity of assumptions of Theorems\;\ref{th:description_PF_V} and\;\ref{th:C1-flows-non-periodic}, and review applications of $\PF$-function to circle actions.

Then in \S\ref{sect:prop_P_func} we describe certain properties of $\PF$-function for continuous flows: local uniqueness, local regularity, and continuity of extensions of regular $\PF$-functions.
We also deduce from well-known M.\;Newman's theorem a sufficient condition for divisibility of regular $\PF$-functions by integers in $\PPF{\Vman}$.
These results will be used in \S\ref{sect:proof:th:description_PF_V} for the proof of Theorem\;\ref{th:description_PF_V}.
\S\ref{sect:diam_orb_lengths} presents a variant of results of M.\;Newman, A.\;Dress, D.\;Hoff\-man, and L.\;N.\;Mann about lower bounds for diameters of orbits of $\ZZZ_p$-actions on manifolds.

\S\ref{sect:Pfunc_nonfixed_pt_set}-\S\ref{sect:unboundedness_of_periods} give sufficient conditions for unboundedness of periods of $\Cont{0}$ and $\Cont{1}$ flows near singular points.
Finally \S\ref{sect:proof:th:C1-flows-non-periodic} contains the proof of Theorem\;\ref{th:C1-flows-non-periodic}.

\section{$\PF$-functions and circle actions}
The following simple statement shows applications of $\PF$-functions to reparametrizations of flows to circle actions.
\begin{lemma}\label{lm:RPF_S1-action}
Let $\theta:\Mman\to\RRR$ be a $\PF$-function on all of $\Mman$.
Then the following map 
\begin{equation}\label{equ:A_x_t_theta}
\BFlow:\Mman\times\RRR\to\Mman,
\qquad
\BFlow(x,t)=\AFlow(x,t\cdot\theta(x))
\end{equation}
is a flow on $\Mman$ such that $\BFlow_{1}=\id_{\Mman}$, so $\BFlow$ factors to a circle action.
\end{lemma}
\begin{proof}
Indeed, $\BFlow_{1}(x) = \AFlow(x,\theta(x))=x$.
\end{proof}

Suppose that there exists a circle action whose orbits coincide with ones of $\AFlow$.
Then every orbit of $\AFlow$ is either periodic or fixed.
Moreover, due to the following well-known theorem of M.\;Newman the set $\FixF$ should be nowhere dense:
\begin{theorem}[M.\;Newman \cite{Newman:QJM:1931}, see also\;\cite{Smith:AM:1941,MontgomerySamelsonZippin:AnnM:1956,Dress:Topol:1969}]
\label{th:Newman_th}
If a compact Lie group effectively acts on a connected manifold $\Mman$, then the set $\FixF$ of fixed points of this action is nowhere dense in $\Mman$ and, by \cite{MontgomerySamelsonZippin:AnnM:1956}, it does not separate $\Mman$.
\end{theorem}

Suppose now $\FixF=\varnothing$, so all points of $\AFlow$ are periodic.
It will be convenient to call $\AFlow$ a \myemph{$P$-flow}.
Then we have a well-defined function 
$$\lambda:\Mman\to(0,+\infty), \qquad \lambda(x)=\Per(x).$$
This function was studied by many authors.
It can be shown that $\lambda$ is lower semicontinuous and the set $B$ of its continuity points is open in $\Mman$, see 
e.g. D.\;Montgomery\;\cite{Montgomery:AJM:1937} and D.\;B.\;A.\;Epstein\;\cite[\S5]{Epstein:AnnMath:1972}.
Thus in the sense of Definition\;\ref{defn:P-func} $\lambda$ is a $\PF$-function on $B$.

There are certain typical situations in which $\lambda$ is discontinuous.

For instance, if $\lambda$ is locally unbounded, then it can not be continuously extended to all $\Mman$.
Say that a $P$-flow $\AFlow$ has property $PB$ (resp. property $PU$) if $\lambda$ is locally bounded (resp. locally unbounded).
Equivalently, if $\AFlow$ is at least $\Cont{1}$, then instead of periods one can consider lengths of orbits with respect to some Riemannian metric on $\Mman$.

It seems that the first example of a $PU$-flow was constructed by G.\;Reeb\;\cite{Reeb:ASI:1952}.
He produced a $\Cont{\infty}$ $PU
$-flow on a non-compact manifold.
Further D.\;B.\;A.\;Epstein\;\cite{Epstein:AnnMath:1972} constructed a \myemph{real analytic} $PU$-flow on a non-compact $3$-manifold, D.\;Sullivan\;\cite{Sullivan:BAMS:1976, Sullivan:PMIHES:1976} a $\Cont{\infty}$ $PU$-flow on a \myemph{compact} $5$-manifold $S^3\times S^1\times S^1$, and 
D.\;B.\;A.\;Epstein and E.\;Vogt\;\cite{EpsteinVogt:AnnMath:1978} a $PU$-flow on a compact $4$-manifold defined by polynomial equations, with the vector field defining the flow given by polynomials, see also E.\;Vogt\;\cite{Vogt:ManuscrMath:1977}.

On the other hand, the following well-known example of Seifert fibrations shows that even if $\lambda$ is discontinuous, then in some cases it can be continuously extended to all of $\Mman$ so that the obtained function is a $\PF$-function.
\begin{example}\label{exmp:Seifert_fibration}\rm
Let $D^2 \subset\CCC$ be the closed unit $2$-disk centered at the origin, $S^1=\partial D^2$ be the unit circle, and $T=D^2\times S^1$ be the solid torus.
Fix $n\geq2$ and define the following flow on $T$:
$$\AFlow:T\times\RRR\to T,
\qquad
\AFlow(z,\tau,t) = (z e^{2\pi i t /k}, \tau e^{2\pi i t}),
$$
for $(z,\tau,t)\in D^2\times S^1 \times \RRR$.
It is easy to see that every $(z,\tau)\in T$ is periodic.
Moreover, $\Per(z,\tau)=k$ if $z\not=0\in D^2$, while $\Per(0,\tau)=1$.
Thus the function $\Per:T^2\to\RRR$ is discontinuous on the central orbit $0\times S^1$, but it becomes even smooth if we redefine it on $0\times S^1$ by the value $k$ instead of $1$.
This new constant function $\theta\equiv k$ is a regular $\PF$-function and $\PPF{T} = \{nk\}_{n\in\ZZZ}$.
\end{example}

Notice that in this example $\AFlow$ is a suspension flow of a periodic homeomorphism $\dif:D^2\to D^2$ being a rotation by $2\pi/k$.

More generally, let $\dif:\Mman\to\Mman$ be a homeomorphism of a connected manifold $\Mman$ such that the corresponding suspension flow $\AFlow$ of $\dif$ on $\Mman\times S^1$ is a $P$-flow.
This is possible if and only if all the points of $\Mman$ are periodic with respect to $\dif$.
D.\;Montgomery\;\cite{Montgomery:AJM:1937} shown that such a homeomorphism is periodic itself.
Let $k$ be the period of $\dif$.
Then the periods of orbits of $\AFlow$ are bounded with $k$, so $\AFlow$ is a $PB$-flow.
Moreover, similarly to Example\;\ref{exmp:Seifert_fibration}, it can be shown that $\PPF{\Mman}=\{nk\}_{n\in\ZZZ}$.

D.\;B.\;A.\;Epstein\;\cite{Epstein:AnnMath:1972} also proved that if $\Mman$ is a compact orientable $3$-manifold then any $\Cont{r}$ $P$-flow with $(1\leq r \leq \omega)$ has property $BP$ and even there exists a $\Cont{r}$ circle action with the same orbits.
In fact he shown that the structure of $\Cont{r}$ foliations $(1\leq r\leq \infty)$ of compact orientable $3$-manifolds, possibly with boundary, is similar to Seifert fibrations described in Example\;\ref{exmp:Seifert_fibration}.

The problem of bounded periods has its counterpart for foliations with all leaves compact.
The question is whether the volumes of leaves are locally bounded with respect to some Riemannian metric, see e.g.\;\cite{Epstein:AIF:1976,EdwardsKennethSullivan:Top:1977, Muller:EnsMath:1979}.
For instance the mentioned above statements for flows can be adopted for foliations.

Let us mention only one result which is relevant with regular $\PF$-functions.

Let $\mathcal{F}$ be a foliation on a manifold $\Mman$ such that all leaves of $\mathcal{F}$ are closed manifolds of dimension $d$.
Then it is shown in R.\;Edwards, K.\;Millett and D.\;Sullivan\;\cite{EdwardsKennethSullivan:Top:1977} that \myemph{the volumes of all leaves are locally bounded whenever there exists a $d$-form $\omega$ such that its integral over any leaf is positive}.
In the case of flows such a $1$-form often appeared for reparametrizations of flows to circle actions.

For instance, W.\;M.\;Boothby \& H.\;C.\;Wang\;\cite{BoothbyWang:AnnMath:1958} considered a contact manifold $\Mman$ with a contact form $\omega$ being \myemph{regular} in the sense of R.\;Palais\;\cite{Palais:MAMS:1957}: let $\AFld$ be the dual vector field for $\omega$; then for each $x\in\Mman$ there exists a neighbourhood $\Uman$ such that $\orb_y\cap\Uman$ is connected for every $y\in\Uman$.
Under this assumption it is shown in\;\cite{BoothbyWang:AnnMath:1958} that $\AFld$ can be reparametrized to a circle action.
Such a connectivity condition is similar to our regularity of $\PF$-functions, see Remark\;\ref{rem:o_cap_V_connected_mu_is_regular}.

Moreover, suppose $\Mman$ is a $\Cont{r}$ manifold with $(3\leq r\leq\infty)$ and $\AFlow$ be a $\Cont{r}$ $P$-flow on $\Mman$.
A.\;W.\;Wadsley\;\cite{Wadsley:JDG:1975}, using a $1$-form with positive integrals along orbits, proved that \myemph{the existence of a $\Cont{r}$ circle action with the same orbits is equivalent to the existence a Riemannian metric on $\Mman$ in which all orbits are geodesic.}
The proof of sufficiency of this result explicitly used regular $\PF$-functions, c.f.\;\cite[Lm.\;4.2]{Wadsley:JDG:1975}.

\smallskip 

Suppose now that $\AFlow$ has no non-closed orbits.
Then it can happen that $\PPF{\Mman\setminus\FixF}=\{n\theta\}_{n\in\ZZZ}$ for some $\PF$-function $\theta:\Mman\setminus\FixF\to\RRR$, while $\PPF{\Mman}=\{0\}$, so $\theta$ can not be continuously extended to all of $\Mman$.

For instance, let $\AFlow:\CCC\times\RRR\to\CCC$ be a $\Cont{\infty}$ flow on $\CCC$ defined by
$$
\AFlow(z,t) = 
\begin{cases}
e^{2\pi i \, t\, |z|^2} \, z, & z\not=\orig, \\
\orig, & z=\orig,
\end{cases}
$$
where $\orig$ is the origin.
Then $\theta=\frac{1}{|z|^2}$ is a $\Cont{\infty}$ $\PF$-function on $\CCC\setminus\orig$ and $\PPF{\CCC\setminus\orig}=\{n\theta\}_{n\in\ZZZ}$.
On the other hand, $\lim\limits_{z\to\orig}\theta(z)=+\infty$, whence $\theta$ can not be extended even to a continuous function on $\CCC$, whence $\PPF{\CCC}=\{0\}$.

We will show that for $\Cont{1}$ flows triviality of $\PPF{\Mman}$ and non-triviality of $\PPF{\Mman\setminus\FixF}$ ``almost always'' appears due to unboundedness of periods of points near $\FixF$, see Theorem\;\ref{th:periods_to_infinity}.

\subsection{Non-regular $\PF$-functions}
The following example shows that on non-open or disconnected sets $\Vman\subset\Mman$ there may exist non-regular $\PF$-functions.
Moreover, it also shows that $\PF$-functions for continuous flows may vanish at fixed points.

\begin{example}\label{exmp:nonreg-P-function}\rm
Let $\AFlow:\CCC\times\RRR\to\CCC$ be a continuous flow on the complex plane $\CCC$ defined by
$$
\AFlow(z,t) = 
\begin{cases}
e^{2\pi i \, t/ |z|^2} \, z, & z\not=\orig, \\
\orig, & z=\orig.
\end{cases}
$$
The orbits of $\AFlow$ are the origin $\orig\in\CCC$ and the concentric circles centered at $\orig$.
Then $\theta=|z|^2$ is a $\PF$-function on $\CCC$ and 
$$\RPF{\CCC} = \PPF{\CCC}=\{n\theta\}_{n\in\ZZZ}.$$

Also notice that $\theta$ is $\Cont{\infty}$ and $\theta(\orig)=0$.
If $\AFlow$ were at least $\Cont{1}$, then due to Theorem\;\ref{th:C1-flows-non-periodic}, $\theta$ would not vanish at $\orig$.

Let $\Vman_i$, $(i=1,2,3)$ be the corresponding subset in $\CCC$ shown in Figure\;\ref{fig:V_123}.
Thus $\Vman_1$ is an open segment, say $(-1,1)$, on the real axis, 
$\Vman_2$ is a union of two closed triangles with common vertex at the origin $\orig$, 
and $\Vman_3$ is union of a triangle with a segment $(-1,0]$ of the real axis intersecting at the origin.
\begin{center}
\begin{figure}[ht]
\includegraphics[height=1.5cm]{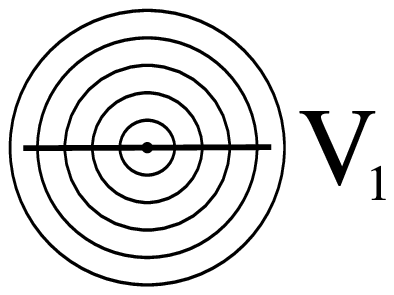}
\qquad
\includegraphics[height=1.5cm]{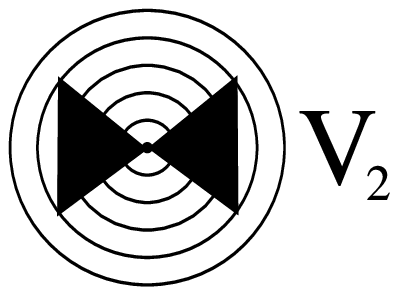}
\qquad
\includegraphics[height=1.5cm]{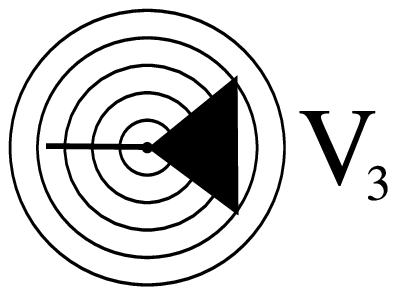}
\label{fig:V_123}
\caption{}
\end{figure}
\end{center}
In particular, $\Int{\Vman_1}=\varnothing$,  $\Int{\Vman_2}$ is not connected, and $\Int{\Vman_3}$ is not dense in $\Vman$.
For any pair $m,n\in\ZZZ$ define the function $\pfunc_{m,n}:\Vman_i\to\RRR$ by
$$
\pfunc_{m,n}(x) =
\begin{cases}
  -m |z|, & \Re(z)\leq 0, \\
  n |z|, & \Re(z)>0.
\end{cases}
$$
Evidently, $\PPF{\Vman_i}=\{\pfunc_{m,n}\}_{m,n\in\ZZZ}$, while $\RPF{\Vman_i}=\{\pfunc_{m,m}\}_{m\in\ZZZ}$.
Thus not every $\PF$-function is regular.
\end{example}

\section{Properties of $\PF$-functions}\label{sect:prop_P_func}

\begin{lemma}\label{lm:small_periods_x_in_FixF}
Let $z\in\Mman$.
Suppose there exists a sequence of periodic points $\{x_i\}_{i\in\NNN}$ converging to $z$ and such that $\lim\limits_{i\to\infty}\Per(x_i)=0$.
Then $z\in\FixF$.
\end{lemma}
\begin{proof}
Suppose $z\not\in\FixF$, so there exists $\tau>0$ such that $z\not=\AFlow_{\tau}(z)$.
Let $\Uman$ be a neighbourhood of $z$ such that 
\begin{equation}\label{equ:Atx_notin_oW}
\Uman \cap \AFlow_{\tau}(\Uman) = \varnothing.
\end{equation}
Since $\AFlow(x,0)=x$, there exists $\eps>0$ and a neighbourhood $\Wman$ of $x$ such that $\AFlow(\Wman\times[0,\eps])\subset\Uman$.
Then we can find $x_i\in\Wman$ with $\Per(x_i)<\eps$.
Hence 
$$\AFlow_{\tau}(x_i)  \in \AFlow_{\tau}(\Uman).$$
On the other hand,
$$\AFlow_{\tau}(x_i) \ \in \ \orb_{x_i} \ = \ \AFlow(x_i,[0,\Per(x_i)]) \ \subset \ \AFlow(\Wman\times[0,\eps]) \ \subset \ \Uman,$$
which contradicts to\;\eqref{equ:Atx_notin_oW}.
\end{proof}

\begin{lemma}[Local uniqueness of $\PF$-functions]\label{lm:local_uniq_sh_func}{\rm c.f.\;\cite[Cor.\;8]{Maks:TA:2003}}
Let $\Vman\subset\Mman$ be \myemph{any subset}, $z\in\Vman\setminus\FixF$ and $\pfunc\in\PPF{\Vman}$.
If $\pfunc(z)=0$, then $\pfunc=0$ on some neighbourhood of $z$ in $\Vman$.
\end{lemma}
\begin{proof}
Suppose $\pfunc$ is not identically zero on any neighbourhood of $z$ in $\Vman\setminus\FixF$.
Then there exists a sequence $\{x_i\}_{i\in\NNN}\subset\Vman\setminus\FixF$ converging to $z$ and such that $\pfunc(x_i)\not=0$.
Hence every $x_i$ is periodic and $\pfunc(x_i)=n_i\Per(x_i)$ for some $n_i\in\ZZZ\setminus\{0\}$.
By continuity of $\pfunc$ we get
$$
0 =
\lim\limits_{i\to\infty}\pfunc(x_i) = 
\lim\limits_{i\to\infty}n_i\Per(x_i). 
$$
Since $|n_i|\geq1$, it follows that $\lim\limits_{i\to\infty}\Per(x_i)=0$, whence by Lemma\;\ref{lm:small_periods_x_in_FixF} $z\in\FixF$, which contradicts to the assumption.
\end{proof}

\begin{lemma}[Local regularity of $\PF$-functions on open sets]\label{lm:local_regularity_P_func}
Let $\Vman\subset\Mman$ be an open subset and $\pfunc\in\PPF{\Vman}$.
Then for each $z\in\Vman$ there exists a neighbourhood $\Wman\subset\Vman$ such that the restriction $\pfunc|_{\Wman}$ is regular.
\end{lemma}
\begin{proof}
Suppose $\pfunc$ is not regular on arbitrary small neighbourhood of $z$.
Then we can find two sequences $\{x_i\}_{i\in\NNN}$ and $\{y_i\}_{i\in\NNN}$ converging to $z$ such that $y_i=\AFlow(x_i,\tau_i)$ for some $\tau_i\in\RRR$ and $\pfunc(x_i)<\pfunc(y_i)$ for all $i\in\NNN$.

It follows that $x_i$ and $y_i$ are periodic.
Otherwise, by Lemma\;\ref{lm:init_prop_ShAV}, we would have $\pfunc(x_i)=\pfunc(y_i)$.
Hence 
$0 < \pfunc(y_i)-\pfunc(x_i) = n_i\Per(x_i) $
for some $n_i\in\ZZZ\setminus\{0\}$.

We claim that $\lim\limits_{i\to\infty}\Per(x_i)=0$.
Indeed, take any $\eps>0$.
Then there is a neighbourhood $\Wman$ of $z$ such that $|\pfunc(y)-\pfunc(x)|<\eps$ for all $x,y\in\Wman$.
Let $N>0$ be such that $x_i,y_i\in\Wman$ for $i>N$, 
$$
\Per(x_i) \leq n_i\Per(x_i) =  \pfunc(y_i) - \pfunc(x_i) < \eps, \qquad i>N.
$$
This implies $\lim\limits_{i\to\infty}\Per(x_i)=0$, whence, by Lemma\;\ref{lm:small_periods_x_in_FixF}, $z\in\FixF$.
But in this case there exists a neighbourhood $\Wman$ of $z$ and $\eps>0$ such that $\AFlow(\Wman\times[0,\eps]) \subset\Vman$.
Take $x_i\in\Wman$ such that $\Per(x_i)<\eps$, then 
$$
\orb_{x_i} = \AFlow(x_i\times[0,\Per(x_i)) \subset \AFlow(\Wman\times[0,\eps]) \subset \Vman.
$$
In other words $\orb_{x_i} \cap \Vman = \orb_{x_i}$ is connected, whence by Lemma\;\ref{lm:init_prop_ShAV} $\pfunc$ is constant on $\orb_{x_i}$.
Therefore $\pfunc(x_i)=\pfunc(y_i)$ which contradicts to the assumption.
\end{proof}

\begin{lemma}[Continuity of extensions of regular $\PF$-functions]\label{lm:extension_of_RPF}
Let $\Vman\subset\Mman$ be an open subset and $\pfunc\in\RPF{\Vman}$ be a regular $\PF$-function on $\Vman$.
Put $\Uman=\AFlow(\Vman\times\RRR)$.
Then $\pfunc$ extends to a $\PF$-function $\tmufunc$ on all of $\Uman$.

If $\Mman$ is a $\Cont{r}$ manifold, $\Vman$ is open in $\Mman$, $\AFlow$ is $\Cont{r}$ on $\Vman\times\RRR$, and $\pfunc$ is $\Cont{r}$ on $\Vman$, then $\tmufunc$ is $\Cont{r}$ on $\Uman$.
\end{lemma}
\begin{proof}
The definition of $\tmufunc$ is evident: if $y\in\Uman$, so $y=\AFlow(x,\tau)$ for some $(x,\tau)\in\Vman\times\RRR$, then we put $\pfunc(y)=\pfunc(x)$.
Since $\pfunc$ is regular, this definition does not depend on a particular choice of such $(x,\tau)$.

It remains to prove continuity of $\tmufunc$ on $\Uman$.
Let $y=\AFlow(x,t)\in\Uman$ for some $(x,t)\in\Vman\times\RRR$.
Since $\Vman$ is open, there exists a neighbourhood $\Wman$ of $y$ in $\Uman$ such that $\AFlow_{-t}(\Wman)\subset\Vman$.
Then $\tmufunc$ can be defined on $\Wman$ by $\tmufunc(z)=\pfunc\circ\AFlow_{-t}(z)$ for all $z\in\Wman$.
This shows continuity of $\tmufunc$ on $\Wman$.

Moreover, if $\Mman$ is a $\Cont{r}$ manifold, $\pfunc$ and $\AFlow$ are $\Cont{r}$, then so is $\tmufunc$.
\end{proof}

\begin{lemma}[Condition of divisibility by integers]
\label{lm:mu_p_is_regular_PF}
Let $\Vman\subset\Mman$ be a connected open subset and $\pfunc:\Vman\to\RRR$ be a regular $\PF$-function.
Suppose that there exist an integer $p\geq2$ and a non-empty open subset $\Wman\subset\Vman$ such that $\AFlow(x,\pfunc(x)/p)=x$ for all $x\in\Wman$, so the restriction of $\pfunc/p$ to $\Wman$ is a $\PF$-function.
Then $\pfunc/p$ is also a $\PF$-function on all of $\Vman$.
\end{lemma}
\begin{proof}
By Lemma\;\ref{lm:extension_of_RPF} we can assume that $\Vman$ is $\AFld$-invariant.
Moreover, it suffices to consider the case when $p$ is a prime.
Define the following map $\dif:\Vman\to\Vman$ by $\dif(x) = \AFlow(x,\pfunc(x)/p)$.
Since $\pfunc$ is constant along orbits of $\AFlow$, it follows that $\pfunc(\dif(x))=\pfunc(x)$, whence
\begin{multline*}
\dif\circ\dif(x) =\AFlow(\dif(x),\pfunc(\dif(x))/p)=  \\ =
\AFlow(\AFlow(x,\pfunc(x)/p,\pfunc(x)/p)= 
\AFlow(x,2\pfunc(x)/p),
\end{multline*}
and so on.
In particular, we obtain that $\dif^p=\id_{\Vman}$, and thus $\dif$ yields a $\ZZZ_p$-action on $\Vman$.
But by assumption this action is trivial on the non-empty open set $\Wman$.
Then by M.\;Newman's Theorem\;\ref{th:Newman_th} the action is trivial on all of $\Vman$, so $\pfunc/p$ is a $\PF$-function on $\Vman$.
\end{proof}

\begin{corollary}\label{cor:mu_0_on_V_IntFix}
Let $\pfunc$ be a regular $\PF$-function on a connected open subset $\Vman\subset\Mman$.
\begin{enumerate}
 \item[(i)]
If $\Vman\cap\Int{\FixF}\not=\varnothing$, then $\pfunc=0$ on $\Vman\setminus\Int{\FixF}$.
 \item[(ii)]
If $\Vman\cap\Int{\FixF}=\varnothing$ and $\pfunc=0$ on some open non-empty subset $\Wman\subset\Vman$, then $\pfunc=0$ on all of $\Vman$.
\end{enumerate}
\end{corollary}
\begin{proof}
Evidently, it suffices to show that in both cases $\pfunc=0$ on $\Vman\setminus\FixF$.

In the case (i) put $\Wman=\Vman\cap\Int{\FixF}$.

Let $p$ be any prime.
Then in both cases $\AFlow(y,\pfunc(y)/p)=y$ for all $y\in\Wman$, where $\Wman$ is a non-empty open set.
Hence by Lemma\;\ref{lm:mu_p_is_regular_PF} $\AFlow(y,\pfunc(y)/p)=y$ for all $y\in\Vman$, that is $\pfunc/p$ is a $\PF$-function on $\Vman$.
Thus if $\pfunc(x)=n\Per(x)\not=0$ for some $x\in\Vman\setminus\FixF$ and $n\in\ZZZ$, then $n$ is divided by $p$.
Since $p$ is arbitrary, we get $n=0$.
\end{proof}

\section{Proof of Theorem\;\ref{th:description_PF_V}}
\label{sect:proof:th:description_PF_V}

\paragraph{\bf (A)}
Suppose $\Int{\FixF}\cap\Vman\not=\varnothing$.
We should prove that the following set
$$
P'=\{\pfunc\in\Cz{\Vman}{\RRR} \ : \ \pfunc|_{\Vman\setminus\FixF}=0 \}.
$$
coincides with $\PPF{\Vman}$.
Evidently, $P'\subset\PPF{\Vman}$.

Conversely, let $\pfunc\in\PPF{\Vman}$.
We claim that \myemph{for every connected component $T$ of $\Vman\setminus\cl{\Int{\FixF}}$ there exists $z\in T$ such that $\pfunc(z)=0$}.
By Lemma\;\ref{lm:local_uniq_sh_func} this will imply that $\pfunc|_{T}=0$.
Since $T$ is arbitrary we will get that $\pfunc=0$ on all of $\Vman\setminus\Int{\FixF}\supset\Vman\setminus\FixF$ and, in particular, that $\pfunc$ is a regular $\PF$-function.

As $\Vman$ is connected, the following set is non-empty, see Figure\;\ref{fig:proof_a}:
$$
B \ := \ \cl{T}\;\cap\;\Vman\;\cap\;(\cl{\Int{\FixF}} \setminus \Int{\FixF}) \ \not= \ \varnothing.
$$
Let $x\in B \subset\Vman=\Int{\Vman}$.
Then by Lemma\;\ref{lm:local_regularity_P_func} there exists an open connected neighbourhood $\Wman$ such that $\pfunc|_{\Wman}$ is a regular $\PF$-function.
Then we have that $\Wman\cap \Int{\FixF}\not=\varnothing$ and $\Wman\cap T\not=\varnothing$ as well.
Since $\pfunc$ is regular on $\Wman$, it follows from (i) of Corollary\;\ref{cor:mu_0_on_V_IntFix} that $\pfunc=0$ on $\Wman\setminus\Int{\FixF}$ and, in particular, on $\Wman\cap T$.

\begin{figure}[ht]
\includegraphics[height=1.5cm]{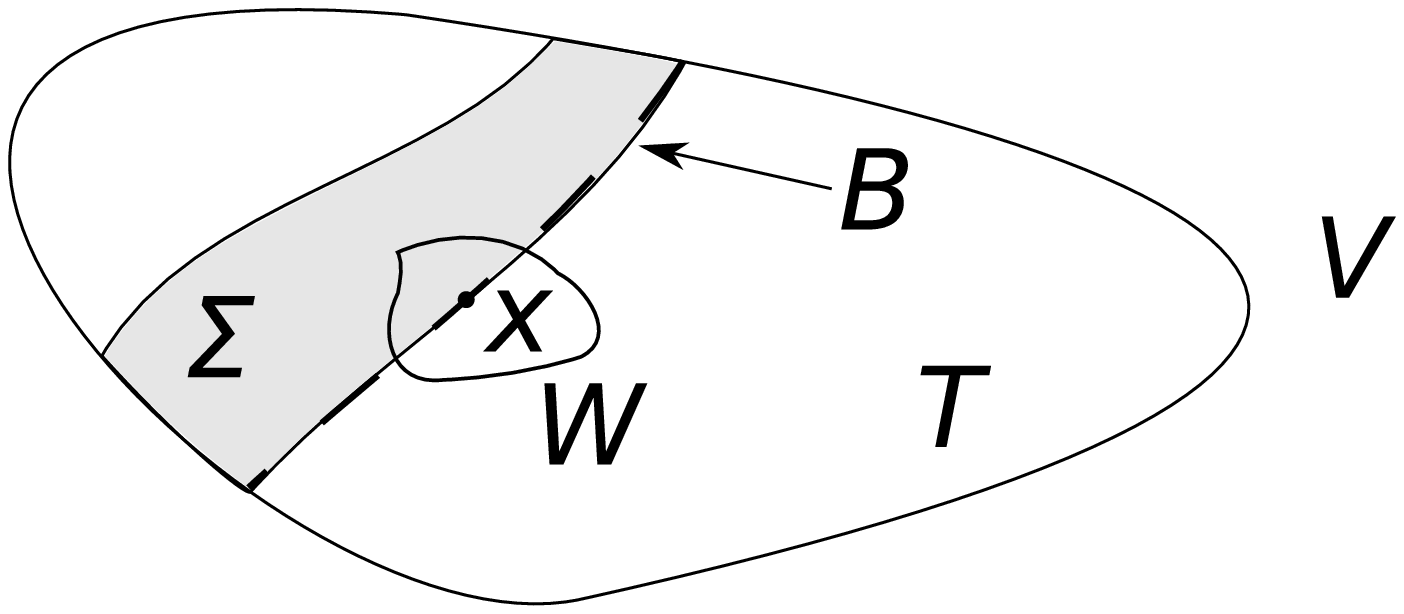}
\label{fig:proof_a}
\caption{}
\end{figure}

\paragraph{\bf (B)}
Suppose that $\Int{\FixF}\cap\Vman=\varnothing$ and $\PPF{\Vman}\not=\{0\}$, so there exists $\pfunc\in\PPF{\Vman}$ which is not identically zero on $\Vman$.
We have to show that $\PPF{\Vman}=\{n\theta\}_{n\in\ZZZ}$ for some $\PF$-function $\theta:\Vman\to\RRR$ satisfying \ThetaPropPositive-\ThetaPropRegularRegCl.

Denote by $Y$ the subset of $\Vman$ consisting of all points $x$ having one of the following two properties:
\begin{enumerate}
\item[(L1)]
$x\in\Vman\setminus\FixF$ and $\pfunc(x)=0$;
\item[(L2)]
$x\in\Vman\cap\FixF$ and there exists a sequence $\{x_i\}_{i\in\NNN}\subset\Vman\setminus\FixF$ converging to $x$ and such that $\pfunc(x_i)=0$ for all $i\in\NNN$.
\end{enumerate}
Evidently, $\pfunc=0$ on $Y$.

\begin{lemma}\label{lm:Y_is_open_closed}
$Y$ is open and closed in $\Vman$.
Hence if $\Vman$ is connected and $\pfunc(x)=0$ for some $x\in\Vman\setminus\FixF$, then $\pfunc=0$ on all of $\Vman$.
\end{lemma}
\begin{proof}
{\em $Y$ is open.}
Let $x\in Y$.
We will show that there exists an open neighbourhood $\Wman$ of $x$ such that $\Wman\subset Y$.

If $x\in\Vman\setminus\FixF$, then, by Lemma\;\ref{lm:local_uniq_sh_func}, $\pfunc=0$ on some neighbourhood $\Wman\subset\Vman\setminus\FixF$ of $x$.
Hence, by (L1), $\Wman\subset Y$.

Suppose $x\in\FixF\cap\Vman \subset \Vman =\Int{\Vman}$.
Then by Lemma\;\ref{lm:local_regularity_P_func} there exists an open neighbourhood $\Wman_x$ of $x$ such that $\pfunc|_{\Wman_x}$ is regular.
We claim that $\Wman_x\subset Y$.

First we show that $\pfunc=0$ on $\Wman_x$.
Indeed, by (L2) there exists a sequence $\{x_i\}_{i\in\NNN}\subset\Vman\setminus\FixF$ converging to $x$ and such that $\pfunc(x_i)=0$ for all $i\in\NNN$.
In particular, $x_i\in\Wman_x$ for some $i\in\NNN$.
Let $C$ be the connected component of $\Wman_x\setminus\FixF$ containing $x_i$.
Then $\pfunc=0$ on an open set $C\cap\Wman_x$, whence, by (ii) of Corollary\;\ref{cor:mu_0_on_V_IntFix}, $\pfunc=0$ on $\Wman_x$.

Therefore $\Wman_x\setminus\FixF \subset Y$.
Let $y\in\Wman_x\cap\FixF$.
Since $\Wman_x\cap\FixF$ is nowhere dense in $\Wman_x$, there exists a sequence $\{y_i\}_{i\in\NNN}\subset\Wman_x\setminus\FixF$ converging to $y$.
But then $\pfunc(y_i)=0$, whence, by (L2), $y\in Y$ as well.

\smallskip 

{\em $Y$ is closed.}
Let $\{x_i\}_{i\in\NNN}\subset Y$ be a sequence converging to some $x\in\Vman$.
We have to show that $x\in Y$.
Since $\pfunc(x_i)=0$, we have $\pfunc(x)=0$ as well.

If $x\in\Vman\setminus\FixF$, then by (L1) $x\in Y$.

Suppose $x\in\Vman\cap\FixF$.
Then we can assume that either $\{x_i\}_{i\in\NNN}\subset\Vman\setminus\FixF$ or $\{x_i\}_{i\in\NNN}\subset\Vman\cap\FixF$.
In the first case $x\in Y$ by (L2).

Suppose $\{x_i\}_{i\in\NNN}\subset\Vman\cap\FixF$.
Since $x_i\in Y$, it follows from (L2) for $x_i$ that there exists a sequence $\{y_{i}^{j}\}_{j\in\NNN}\subset\Vman\setminus\FixF$ converging to $x_i$ and such that $\pfunc(y_{i}^{j})=0$.
Then for each $i\in\NNN$ we can find $n(i)\in\NNN$ such that the \myemph{diagonal} sequence $\{y_{i}^{n(i)}\}_{i\in\NNN} \subset \Vman\setminus\FixF$, converges to $x$, and satisfies $\pfunc(y_{i}^{n(i)})=0$.
Hence, by (L2), $x\in Y$.
\end{proof}

Thus we can assume that $\pfunc\not=0$ on $\Vman\setminus\FixF$.
In particular, all points in $\Vman\setminus\FixF$ are periodic.

Take any $x\in\Vman\setminus\FixF$ and consider the following \myemph{homomorphism} 
$$
e_x:\PPF{\Vman}\to\ZZZ,
\qquad 
e_x(\nufunc)=\nufunc(x)/\Per(x),
$$
for $\nufunc\in\PPF{\Vman}$.
If $\nufunc(x)=0$, then, as noted above, $\nufunc=0$ on all of $\Vman$, whence \myemph{$e_x$ is a monomorphism}.
Moreover, $e_x(\pfunc)=\pfunc(x)\not=0$, whence $e_x$ yields an isomorphism of $\PPF{\Vman}$ onto a non-zero subgroup $k\ZZZ$ of $\ZZZ$ for some $k\in\NNN$.
Put $\theta=e_x^{-1}(k)$.
Then $\PPF{\Vman}=\{n\theta\}_{n\in\ZZZ}$.

\smallskip 

It remains to verify properties of $\theta$.

\smallskip 

\paragraph{\bf \ThetaPropPeriod$\Rightarrow$\ThetaPropPositive}
We have that $\theta(x)=\Per(x)>0$ on an open and everywhere dense subset $\Qman\subset\Vman$, whence $\theta\geq0$ on $\Vman$.
On the other hand, by Lemma\;\ref{lm:Y_is_open_closed}, $\theta\not=0$ on $\Vman\setminus\FixF$, whence $\theta>0$ on $\Vman\setminus\FixF$.

\smallskip 

\paragraph{\bf \ThetaPropPeriod$\Rightarrow$\ThetaPropRegularOpen}
We have to show that $\theta$ is regular, that is $$\theta(x)=\theta(\AFlow_{\tau}(x))$$
for any $x\in\Vman\setminus\FixF$ and $\tau\in\RRR$ such that $\AFlow_{\tau}(x)\in\Vman$.

First notice that for any open subsets $A,B \subset\Mman$ we have that 
\begin{equation}\label{equ:cl_UcapclV__cl_UcapV}
\cl{A\cap\cl{B}} \ = \ \cl{\cl{A}\cap B} \ = \ \cl{A\cap B}.
\end{equation}

Since $\Qman$ is open and everywhere dense in $\Vman$, it follows that 
\begin{multline*}
\AFlow_{\tau}(x) 
\;\in\;
\Vman\cap\AFlow_{\tau}(\Vman) 
\;\subset\;
\cl{\cl{\Qman}\cap\AFlow_{\tau}(\Vman)} 
\;\stackrel{\eqref{equ:cl_UcapclV__cl_UcapV}}{=\!=\!=}\; \\ =\;
\cl{\Qman\cap\cl{\AFlow_{\tau}(\Vman)}} \;=\; 
\cl{\Qman\cap\cl{\AFlow_{\tau}(\Qman)}} 
\;\stackrel{\eqref{equ:cl_UcapclV__cl_UcapV}}{=\!=\!=}\;
\cl{\Qman\cap\AFlow_{\tau}(\Qman)}.
\end{multline*}

In other words, there exists a sequence $\{x_i\}_{i\in\NNN}\subset \Qman$ converging to $x$ and such that $\{\AFlow_{\tau}(x_i)\}_{i\in\NNN} \subset  \Qman$.
Then $\theta(\AFlow_{\tau}(x_i)) = \theta(x_i) = \Per(x_i)$.
Whence
$$\theta(\AFlow_{\tau}(x)) =
\lim\limits_{i\to\infty} \theta(\AFlow_{\tau}(x_i)) =
\lim\limits_{i\to\infty} \theta(x_i) =
\theta(x).
$$

\smallskip 

\paragraph{\bf \ThetaPropRegularOpen$\Rightarrow$\ThetaPropCircleAction}
See Lemma\;\ref{lm:extension_of_RPF}.

\smallskip 

\paragraph{\bf \ThetaPropPeriod}
The proof consists of the following three statements.

\begin{clm}
Let $x\in\Vman\setminus\FixF$.
Then there exist an open connected neighbourhood $\Wman_x$ of $x$ in $\Vman$, 
a regular $\PF$-function $\theta_x\in\PPF{\Wman_x}$, a number $m_x\in\ZZZ\setminus\{0\}$, and 
an open and everywhere dense subset $\Qman_x \subset\Wman_x$ consisting of periodic points such that 
\begin{enumerate}
 \item[(a)]
$\PPF{\Wman_x}=\{m\theta_x\}_{m\in\ZZZ}$,
 \item[(b)]
$\theta=m_x\theta_x$ on $\Wman_x$,
 \item[(c)] 
$\theta_x(y)=\Per(y)$ for all $y\in\Qman_x$.
\end{enumerate}
\end{clm}
\begin{proof}
By Lemma\;\ref{lm:local_regularity_P_func} there exists an open connected neighbourhood $\Wman_x$ of $x$ such that $\cl{\Wman_x}\subset\Vman\setminus\FixF$ and $\theta|_{\Wman_x}$ is regular.
Notice that if we decrease $\Wman_x$, then the restriction of $\theta$ to $\Wman_x$ remains regular.
Therefore we can additionally assume that there exists $\eps\in(0,\Per(x))$ such that
\begin{enumerate}
\item[(i)] 
$\theta(y)<\theta(x)+\eps$ for all $y\in\cl{\Wman_x}$;
\item[(ii)]
$\Per(x)< \Per(y)+\eps$ for all $y\in\cl{\Wman_x}$;
\item[(iii)]
there is $N>0$ such that $n_y := \theta(y)/\Per(y) < N$ for all $y\in\cl{\Wman_x}$.
\end{enumerate}
Indeed, (i) follows from continuity of $\theta$, and (ii) from \myemph{lower semicontinuity} of $\Per$, c.f.\;\cite{Montgomery:AJM:1937}.

More precisely, suppose (ii) fails.
Then there exists a sequence $\{x_i\}_{i\in\NNN}\subset\Vman\setminus\FixF$ converging to $x$ and such that $\Per(x)\geq \Per(x_i)+\eps$.
In particular, periods of $x_i$ are bounded above and we can assume that $\lim\limits_{i\to\infty}\Per(x_i)=\tau<\infty$ for some $\tau$.
Then 
\begin{equation}\label{equ:Perx_g_tau}
\Per(x)\geq\tau + \eps>\tau.
\end{equation}
But $\AFlow(x,\tau) = \lim\limits_{i\to\infty}\AFlow(x_i,\Per(x_i)) = x$, so $\tau=n\Per(x) \geq \Per(x)$ for some $n\in\NNN$, which contradicts to\;\eqref{equ:Perx_g_tau}.
This proves (ii).

To establish (iii) notice that it follows from (i) and (ii) that 
$$
n_y(\Per(x)-\eps) < n_y \Per(y) = \theta(y) < \theta(x)+\eps,
$$
whence 
$$
N := \frac{\theta(x)+\eps}{\Per(x)-\eps} > n_y.
$$
This proves (iii).

\smallskip 

Consider the group $\PPF{\Wman_x}$.
As $\Wman_x$ is open and connected, we have that $\PPF{\Wman_x}=\{m\theta_x\}_{m\in\ZZZ}$ for some $\theta_x\in\Cz{\Wman}{\RRR}$.
By assumption, $\theta$ is a $\PF$-function on $\Wman_x$, whence $\theta|_{\Wman_x} = m_x\theta_x$ for some $m_x\in\ZZZ\setminus\{0\}$.

\smallskip 

To construct $\Qman_x$ notice that for each $y\in\Wman_x\setminus\FixF$ there exists a unique $n_y\in\ZZZ$ such that $\theta_x(y)=n_y\Per(y)$.
For every $n\in\NNN$ denote by $T_{n}$ the subset of $\Wman_x$ consisting of all $y$ such that $n_y$ is divided by $n$.
Since the values $n_y$ are bounded above, it follows that $T_n$ is non-empty only for finitely many $n$.
Also notice that 
$$
\Wman_x\setminus\FixF=\mathop\cup\limits_{n=1}^{N} T_n.
$$
We claim that \myemph{$\cl{T_n}$ is nowhere dense for $n\geq 2$}.
Indeed, suppose $\Int{\cl{T_n}}\not=\varnothing$.
Then $\theta_x/n$ is a regular $\PF$-function on $\Int{\cl{T_n}}$ and therefore, by Lemma\;\ref{lm:mu_p_is_regular_PF}, on all of $\Wman_x$.
However this is possible only for $n=1$ as $\theta_x$ generates $\PPF{\Wman_x}$.
Thus the subset $Q_x := \Int{\cl{T_1}}\cap\Wman_x$ is open and everywhere dense in $\Wman$ and $\theta(y)=\Per(y)$ for all $y\in\Qman_x$.
\end{proof}

\begin{clm}\label{clm:thetax_thetay__mx_my}
Let $x,y\in\Vman\setminus\FixF$.
Then $\theta_x=\theta_y$ on $\Wman_x\cap\Wman_y$ and $m_x=m_y$.
\end{clm}
\begin{proof}
Indeed, since $\Qman_x$ ($\Qman_y$) is open and everywhere dense in $\Wman_x$ ($\Wman_y$), it follows that $\Qman_x\cap\Qman_y$ is open and everywhere dense in $\Wman_x\cap\Wman_y$.
Moreover, for each $z\in \Qman_x\cap\Qman_y$ we have that $\theta_x(z)=\theta_y(z)=\Per(z)$.
Then by continuity $\theta_x=\theta_y$ on $\Wman_x\cap\Wman_y$.

In particular, if $z\in\Qman_x\cap\Qman_y$, then $\theta(z)=m_x\Per(z)=m_y\Per(z)$, whence $m_x=m_y$.
\end{proof}

Let $T$ be a path component of $\Vman\setminus\FixF$.
Then by Claim\;\ref{clm:thetax_thetay__mx_my} $m_x$ is the same for all $x\in T$ and we denote their common value by $m_{T}$.
It also follows that the functions $\{\theta_{x}\}_{x\in T}$ define a continuous function $\theta_{T}:T\to\RRR$ such that 
$\theta|_{T} = m_{T}\,\theta_{T}$.
Thus if we put $\Qman_{T} = \mathop\cup\limits_{x\in T}\Qman_{x}$, then $\Qman_{T}$ is open and everywhere dense in $T$ and $\theta_{T}(y)=\Per(y)$ for all $y\in\Qman_{T}$.

\begin{clm}\label{clm:thetaT_thetaS__mT_mS}
Let $S$ and $T$ be any connected components of $\Vman\setminus\FixF$ such that $\cl{S}\cap\cl{T}$.
Then $m_{S}=m_{T}$.
\end{clm}
\begin{proof}
We can assume that $T\not=S$.
Let $x\in\cl{S}\cap\cl{T}\subset\Vman\cap\FixF$ and $\Wman_x$ be an open, connected neighbourhood of $x$ in $\Vman$ such that $\theta|_{\Wman_x}$ is a regular $\PF$-function on $\Wman_x$.
Notice that $\theta_S=\theta/m_S$ is a regular $\PF$-function on the non-empty open set $\Wman_x\cap S$, whence, by Lemma\;\ref{lm:mu_p_is_regular_PF}, $\theta/m_S$ is a $\PF$-function on all of $\Wman_x$.

If $x\in \Qman_T\cap\Wman_x$, then $\theta(x) = m_{T}\theta_{T}(x)=m_{T}\Per(x)$, therefore $m_{T}$ is divided by $m_{S}$.
By symmetry $m_{S}$ is divided by $m_{T}$ as well, whence $m_{S}=m_{T}$.
\end{proof}

Since $\Vman$ is connected, it follows from Claim\;\ref{clm:thetaT_thetaS__mT_mS} that the number $m_{T}$ is the same for all connected components $T$ of $\Vman\setminus\FixF$.
Denote the common value of these numbers by $m$.
Then $\theta/m$ is continuous on $\Vman$ and $\AFlow(x,\theta(x)/m)=x$ for all $x\in\Vman$.
Since $\theta$ generates $\PPF{\Vman}$, we obtain that $m=1$.

Let $\Qman$ be the union of all $\Qman_{T}$, where $T$ runs over the set of all connected components of $\Vman\setminus\FixF$.
Since for every such component $T$ we have that $\theta=m\theta_{T}=\theta_{T}$ on $T$, it follows that $\theta(x)=\Per(x)$ for all $x\in\Qman_{T}$.
Theorem\;\ref{th:description_PF_V} is completed.

\section{Diameters and lengths of orbits} \label{sect:diam_orb_lengths}

\subsection{Effective $\ZZZ_{p}$-actions}
\label{sect:effect_Zp_actions}
We recall here results of A.\;Dress\;\cite[Lm.\;3]{Dress:Topol:1969} and D.\;Hoffman and L.\;N.\;Mann\;\cite[Th.\;1]{HoffmanMann:PAMS:76} about diameters of orbits of effective $\ZZZ_p$-actions.

For $x,y\in\RRR^n$ denote by $d(x,y)$ the usual Euclidean distance, and by $B_{r}(x)$, $(r>0)$ the open ball of radius $r$ centered at $x$.

Let $\Wman$ be an open subset of the half-space $\RRR^{n}_{+}=\{x_n\geq0\}$ and $x\in\Wman$.
Define the \myemph{radius $r_x$ of convexity of $\Wman$ at $x$} in the following way.
If $x\in\Int{\RRR^{n}_{+}}\cap\Wman$, then 
$$
r_x = \sup \{ r>0 \ : \ B_r(x)\subset\Wman\}
$$
Otherwise, $x\in\partial\RRR^{n}_{+}\cap\Wman$ and we put
$$
r_x = \sup \{ r>0 \ : \ (B_r(x)\cap\RRR^{n}_{+})\subset\Wman\}.
$$

\begin{lemma}\label{lm:Dress:Topol:1969:Lm_3}{\rm\;\cite[Lm.\;3]{Dress:Topol:1969}}
Let $U\subset\RRR^n$ be an open, relatively compact and connected subset, $p$ be a prime, and $\dif:\cl{U}\to\cl{U}$ be a homeomorphism which induces a non-trivial $\ZZZ_{p}$-action, that is $\dif\not=\id_{\cl{U}}$ but $\dif^{p}=\id_{\cl{U}}$.
Define two numbers:
$$
\begin{array}{rcl}
D(U) & = & \max\{ \min\{d(x,y)  \ : \ y\in\cl{U}\} \ : \ x \in U  \}, \\ [2mm]
C(U) & = & \max\{ d(x,\dif^{a}(x))  \ : \ a=0,\ldots,p-1, \, x\in\cl{U}\setminus\Int{\cl{U}}\}.
\end{array}
$$
Then $D(U) < C(U)$.
\end{lemma}

The next Lemma\;\ref{lm:HoffmanMann:diameters_of_orbits} is a variant of\;\cite[Th.\;1]{HoffmanMann:PAMS:76}.
It seems that in the proof of\;\cite[Th.\;1]{HoffmanMann:PAMS:76} the condition of connectedness of the set $U$, see\;\eqref{equ:cup_hi_Bs} below, is missed, c.f.\! paragraph after the assumption (H) on\;\cite[page 345]{HoffmanMann:PAMS:76}.
Therefore we recall the proof which is also applicable to manifolds with boundary.
\begin{lemma}\label{lm:HoffmanMann:diameters_of_orbits}
{\rm c.f.\;\cite[Th.\;1]{HoffmanMann:PAMS:76}}
Let $\Wman\subset\RRR^{n}_{+}=\{x_n\geq0\}$ be an open subset, $p$ be a prime, and $\dif:\cl{\Wman}\to\cl{\Wman}$ be a homeomorphism which induces a nontrivial $\ZZZ_p$-action.
Suppose $\dif(z)=z$ for some $z\in\Wman$.
Let also $r_z$ be the radius of convexity of $\Wman$ at $z$ and $r\in(0,r_z)$.

If $z\in\Int{\RRR^{n}_{+}}$, then there exist $x\in\partial B_{r/2}(z)$ and $a=1,\ldots,p-1$ such that 
$$
d(z,x) \ \leq \ 2 \cdot d(x,\dif^{a}(x)).
$$

If $z\in\partial\RRR^{n}_{+}$, then there exist $x\in\partial (B_{2r/3}(z)\cap\Wman)$ and $a=1,\ldots,p-1$ such that 
$$
d(z,x) \ \leq \ 4\cdot d(x,\dif^{a}(x)).
$$
\end{lemma}
\begin{proof}
For simplicity denote $B_{r}(z)$ by $B_r$.

1) First suppose that $z\in\Int{\RRR^{n}_{+}}$.
For each $s\in(0,r_z)$ put
\begin{equation}\label{equ:cup_hi_Bs}
U_{s} \;=\; B_{s} \;\cup\; \dif(B_{s}) \;\cup\; \cdots \;\cup\; \dif^{p-1}(B_{s}).
\end{equation}
Then $U_{s}$ is open, relatively compact, and $\dif$ yields a non-trivial $\ZZZ_p$-action on $\cl{U}$.
Moreover, by assumption $\dif(z)=z$, therefore $U_{s}$ is connected.
Then by Lemma\;\ref{lm:Dress:Topol:1969:Lm_3} 
$$ D(U_{s}) \leq C(U_{s}). $$

Notice that $B_{s} \subset \cl{U_{s}}$, whence $s \leq D_{s}$.

On the other hand, suppose
\begin{equation}\label{equ:assumption_H}
d(y,\dif^{a}(y)) < r-s, \ \ \text{for all $y\in\partial\cl{U_s}$ and $a=1,\ldots,p-1$}.
\end{equation}
Then in particular, $C(U_{s})<r-s$ and thus
$$
s \ \leq \ D(U_{s}) \ \leq \ C(U_{s}) \ < \ r-s,
$$
whence $s<r/2$.

Thus if $s=r/2$, then\;\eqref{equ:assumption_H} fails, whence there exist $y\in\partial\cl{U_{r/2}}$ and $b\in\{0,\ldots,p-1\}$ such that $d(y,\dif^{b}(y))\geq r-r/2=r/2$.
However $$\partial\cl{U_{r/2}} \  \subset \ \mathop\cup\limits_{i=0}^{p-1} \dif^{i}(\partial B_{r/2}),$$
whence $y=\dif^{c}(x)$ for some $x\in\partial B_{r/2}$.
Therefore at least one of the distances $d(x,y)$ or $d(x,\dif^{b}(y))$ is not less than $r/4$.
In other words, $d(x,\dif^{a}(x))\geq r/4$ for some $x\in\partial B_{r/2}$ and $a\in\{1,\ldots,p-1\}$.
Then
$$
d(x,z) \ = \ \tfrac{r}{2} \ = \ 2 \cdot \tfrac{r}{4} \ \leq \ 2 \cdot d(x,\dif^{a}(x)).
$$

\smallskip

2) Let $z\in\partial\RRR^{n}_{+}$.
For each $s\in(0,r_z)$ let $A_{s} = B_{s} \cap \Int{\partial\RRR^{n}_{+}}$ be the open upper half-disk centered at $z$, and 
$$
U'_{s} \;=\; A_{s} \;\cup\; \dif(A_{s}) \;\cup\; \cdots \;\cup\; \dif^{p-1}(A_{s}).
$$
Then $U'_{s}$ is open, relatively compact, and $\dif$ yields a non-trivial $\ZZZ_p$-action on $\cl{U'_{s}}$.
Moreover, it is easy to see that $U'_{s}$ is connected, whence by Lemma\;\ref{lm:Dress:Topol:1969:Lm_3} 
$$ D(U'_{s}) \leq C(U'_{s}). $$

Moreover, $B_{s/2} \subset \cl{U'_{s}}$.
Therefore $s/2 \leq D(U'_{s})$.
Hence if we suppose that $C(U'_{s})<r-s$, then $s/2 < r-s$ and thus $s < 2r/3$.

Put $s=2r/3$.
Then there exists $y \in \partial U'_{2r/3}$ and $b\in\{1,\ldots,p-1\}$ such that $d(y,\dif^{b}(y))>r-2r/3=r/3$.

Again $\partial\cl{U'_{s}} \  \subset \ \mathop\cup\limits_{i=0}^{p-1} \dif^{i}(\partial A_{s})$, whence we can find $x\in\partial A_{2r/3}$ such that $d(x,\dif^{a}(x))>r/6$ for some $a\in\{1,\ldots,p-1\}$.
Then
$$
d(x,z) \ \leq \ \tfrac{2r}{3} \ = \ 4 \cdot \tfrac{r}{6} \ \leq \ 4 \cdot d(x,\dif^{a}(x)).
$$ 
Lemma is proved.
\end{proof}

\subsection{Periodic orbits}\label{sect:diam_periodic_orb}
Let $\AFld$ be a $\Cont{1}$ vector field on a manifold $\Mman$ and $d$ be any Riemannian metric on $\Mman$.
Then for any periodic point $x$ of $\AFld$ the length $l(x)$ of its orbit can be calculated as follows:
\begin{equation}\label{equ:orb_length}
 l(x) = \int\limits_{0}^{\Per(x)}\|\AFld(\AFlow(x,t))\|dt.
\end{equation}
Hence 
\begin{equation}\label{equ:l_estimate_Per_F}
l(x) \;\leq\; \Per(x) \cdot \sup_{t\in[0,\Per(x)]} \|\AFld(\AFlow(x,t))\|.
\end{equation}

Also notice that 
\begin{equation}\label{equ:diam_and_len}
2\cdot\diam(\orb_{x}) \leq l(x).
\end{equation}
Indeed, let $y,z\in\orb_{x}$ be points for which $d(y,z)=\diam(\orb_{x})$.
These points divide $\orb_x$ into two arcs each of which has the length $\geq d(y,z)$.
This implies\;\eqref{equ:diam_and_len}.

\section{$\PF$-functions on the set of non-fixed points}
\label{sect:Pfunc_nonfixed_pt_set}
Let $\Vman\subset\Mman$ be an open subset, $\pfunc:\Vman\setminus\FixF\to\RRR$ be a $\PF$-function and $\alpha\in\RRR$.
Define the following map $\dif_{\alpha}:\Vman\to\Mman$ by:
\begin{equation}\label{equ:h_F_x_mup}
\dif_{\alpha}(x) = 
\begin{cases}
\AFlow(x,\alpha\,\pfunc(x)), & x\in\Vman\setminus\FixF, \\
x, & x\in\FixF\cap\Vman.
\end{cases}
\end{equation}
Then $\dif_{\alpha}$ is continuous on $\Vman\setminus\FixF$ but in general it is discontinuous at points of $\FixF\cap\Vman$.

The aim of this section is to establish implications between the following five statements:
\begin{enumerate}
\item[$(A)$]
The periods of periodic points in $\Vman\setminus\pfunc^{-1}(0)$ are uniformly bounded above with some constant $\upb>0$, that is for each $x\in\Vman$ with $\pfunc(x)\not=0$ we have that $\Per(x)<\upb$.
\item[$(B)$]
Every $z\in\FixF\cap\Vman$ has a neighbourhood $\Wman\subset\Vman$ such that $\pfunc$ is \myemph{regular} on $\Wman\setminus\FixF$, that is for every $y\in\Wman\setminus\FixF$ the restriction of $\pfunc$ to $\orb_{y}\cap\Wman$ is constant.
\item[$(C)_{\alpha}$]
The map $\dif_{\alpha}$ is continuous on all of $\Vman$.
\end{enumerate}

Let $z\in\FixF\cap\Vman$.

\begin{enumerate}
\item[$(D)_{\alpha}$]
Suppose $\alpha=q/p\in\QQQ$, where $q\in\ZZZ$ and $p\in\NNN$.
There exists a neighbourhood $\Wman\subset\Vman$ of $z$ such that $\dif_{\alpha}(\Wman)=\Wman$, the restriction $\dif_{\alpha}:\Wman\to\Wman$ is a homeomorphism, and $\dif_{\alpha}^{p}=\id_{\Wman}$.

\item[$(E)$]
There exist $\MULT>0$, a Euclidean metric $d$ on some neighbourhood $\Wman$ of $z$, and a sequence $\{x_i\}_{i\in\NNN} \subset \Vman\setminus\FixF$ converging to $z$ such that $\pfunc(x_i)\not=0$ and 
$$d(z,x_i) \ < \ \MULT \cdot \diam(\orb_{x_i}\cap\Wman).$$
\end{enumerate}

\begin{remark}\rm
By a \myemph{Euclidean} metric on $\Wman$ in $(E)$ we mean a metric induced by some embedding $\Wman\subset\RRR^n$.
In fact, this condition can be formulated for arbitrary Riemannian metrics, but for technical reasons (see especially Lemma\;\ref{lm:estim_for_F}) we restrict ourselves to Euclidean ones.
\end{remark}

\begin{remark}\rm
If $z\in\FixF\cap\Vman$ has a neighbourhood $\Wman\subset\Vman$ such that $\Wman\setminus\FixF$ is connected, then by Theorem\;\ref{th:description_PF_V} condition $(B)$ holds for $z$.
\end{remark}

\begin{lemma}\label{lm:prop_h_mu_A_BC_D}
The following implications hold true:
$$ {(A)}  \quad  \Rightarrow  \quad  {(B)}\,\&\,{(C)_{\alpha}}. $$

If $\alpha\in\QQQ$, then for every $z\in\FixF\cap\Vman$ 
$$ {(B)}\,\&\,{(C)_{\alpha}} \quad \Rightarrow \quad {(D)_{\alpha}}. $$

If $\dif$ is not the identity on some neighbourhood of $z\in\FixF\cap\Vman$, then 
$$ {(D)_{\alpha}} \quad \Rightarrow \quad {(E)}, $$
and we can take $\MULT=4$ in $(E)$. 
\end{lemma}
\begin{proof}
$(A)\Rightarrow(B)$.
Since $z\in\FixF\cap\Vman$ there exists a neighbourhood $\Wman$ of $z$ such that $\AFlow(\Wman\times[0,c])\subset\Vman$.
Then $\Wman$ satisfies $(B)$.
Indeed, let $y\in\Wman\setminus\FixF$.
We have to show that $\pfunc|_{\orb_y\cap\Wman}$ is constant.

If $\pfunc=0$ on $\orb_y\cap\Wman$ there is nothing to prove.
Therefore we can assume that $\pfunc(y)\not=0$. 
Then, by $(A)$, $\Per(y)<\upb$, whence
$$
\orb_y \ = \
\AFlow(y\times[0,\Per(y)]) \ = \
\AFlow(y\times[0,\upb]) \ \subset \
\AFlow(\Wman\times[0,\upb]) \ \subset \
\Vman.
$$
Thus $\orb_y\cap\Vman=\orb_y$ is connected.
Then, by Lemma\;\ref{lm:init_prop_ShAV}, $\pfunc$ is constant along $\orb_y$ and therefore on $\orb_y\cap\Wman$.

\smallskip 

$(A)\Rightarrow(C)_{\alpha}$.
It suffices to show that $\dif_{\alpha}$ is continuous at each $z\in\FixF\cap\Vman$.
Let $\Vman'\subset\Vman$ be any neighbourhood of $z$ and $\Wman$ be another neighbourhood of $z$ such that $\AFlow(\Wman\times[0,c])\subset\Vman'$.
We claim that $\dif_{\alpha}(\Wman)\subset\Vman'$.
This will imply continuity of $\dif_{\alpha}$ at $z$.

Let $x\in\Wman$.
If $x\in\FixF\cap\Wman$ or $\pfunc(x)=0$, then $\dif_{\alpha}(x)=x\in\Wman\subset\Vman'$.
Otherwise, $\pfunc(x)\not=0$ and $x$ is periodic.
Hence $\dif_{\alpha}(x) = \AFlow(x,\tau)$ for some $\tau\in[0,\Per(x)] \subset[0,c]$.
Therefore $\dif_{\alpha}(x) \in\AFlow(\Wman\times[0,c])\subset\Vman'$.

\smallskip 

$(B)\&(C)_{\alpha}\Rightarrow(D)_{\alpha}$.
We have that $\alpha=q/p$, where $q\in\ZZZ$ and $p\in\NNN$.
For simplicity denote $\dif_{\alpha}$ by $\dif$.
Since $\dif(z)=z$ and $\dif$ is continuous, there exists a neighbourhood $B$ of $z$ such that $\dif^i(\Bman)\subset\Vman$ for all $i=0,\ldots,p$.
Denote
$$
\Wman = \Bman\cup \dif(\Bman) \cup \cdots \cup \dif^{p-1}(\Bman).
$$
We claim that $\Wman$ satisfies (D).

Indeed, let $x\in\Vman$ and suppose that $\dif(x)\in\Vman$ as well.
Since $\dif(x)$ belongs to the orbit of $x$, then, by $(B)$, $\pfunc(x)=\pfunc(\dif(x))$.
Hence 
\begin{multline}\label{equ:h2_Fx2mup}
\dif^{2}(x) = 
\AFlow\bigl(\dif(x),\alpha\cdot\pfunc(\dif(x))\bigr) = 
\AFlow\bigl(\dif(x),\alpha\cdot\pfunc(x)\bigr) = \\ =
\AFlow\bigl(\AFlow(x,\alpha\cdot\pfunc(x)),\alpha\cdot\pfunc(x)\bigr) = 
\AFlow\bigl(x,\, 2\alpha\cdot\pfunc(x)\bigr).
\end{multline}
By induction we will get that if $\dif^i(x)\in\Vman$ for all $i=0,\ldots,j-1$, then
$$
\dif^{j}(x) = \AFlow\bigl(x,\, j\alpha\cdot\pfunc(x)\bigr).
$$
In particular, $\dif^{kp}(x) = \AFlow(x,kq\pfunc(x))=x$ for any $k\in\ZZZ$.
By $(C)_{\alpha}$ we have that $\dif$ is continuous on $\Vman$, whence $\dif$ yields a homeomorphism of $\Wman$ onto itself and $\dif^p|_{\Wman}=\id_{\Wman}$.

\smallskip 

$(D)_{\alpha}\Rightarrow(E)$.
Again denote $\dif_{\alpha}$ by $\dif$.
Since $\dif$ is not the identity on $\Wman$, we can assume that $p$ is a prime and thus the action of $\dif$ is effective.
This can be done by replacing $\dif$ with $\dif^{n}$ for some $n\in\NNN$ such that $p/n$ is a prime.

Decreasing $\Wman$ we can assume that $\Wman$ is an open subset of the half-space $\RRR^n_{+}=\{x_n\geq0\}$.
Let $d$ be the corresponding Euclidean metric on $\Wman$ and $r_z$ be the radius of convexity of $\Wman$ at $z$.
Then by Lemma\;\ref{lm:HoffmanMann:diameters_of_orbits} for each $r\in(0,r_z)$ there exist $x_{r}\in\Wman$ and $a_{r}\in\{1,\ldots,p-1\}$ such that 
$$
d(z,x_r)  \ \leq \
\MULS \cdot r \ \leq \
\MULT \cdot d(x_{r},\dif^{a_{r}}(x_{r})) \ \leq \
\MULT \cdot \diam(\orb_{x_{r}}),
$$
for some $\MULS,\MULT>0$. 
In fact, $\MULS=\tfrac{1}{2}$ and $\MULT=2$ if $z\in\Int{\Mman}$, and $\MULS=\tfrac{2}{3}$ and $\MULT=4$ if $z\in\partial\Mman$.
Notice that $a_{r}$ may take only finitely many values.
Therefore we can find $a\in\{1,\ldots,p-1\}$ and a sequence $\{x_{r_i}\}_{i\in\NNN}$ such that $\lim\limits_{i\to\infty}r_i=0$ and $a_{r_i}=a$ for all $i\in\NNN$.
\end{proof}

\section{Condition $(E)$ for $\Cont{1}$ flows}\label{sect:cond_E_for_C1_flows}
Condition $(E)$ defined in the previous section gives some lower bound for diameters of orbits of a sequence $\{x_i\}_{i\in\NNN}$ of periodic points converging to a fixed point $z$.
In this section it is shown that for a $\Cont{1}$ flow that condition allows to estimate periods of $x_i$.

Let $\Mman$ be a $\Cont{r}$, $(r\geq1)$, connected, $m$-dimensional manifold possibly non-compact and with or without boundary.
Let also $\AFld$ be a $\Cont{r}$ vector field on $\Mman$ tangent to $\partial\Mman$ and generating a $\Cont{r}$-flow $\AFlow:\Mman\times\RRR \to \Mman$.
Again by $\FixF$ we denote the set of fixed points of $\AFlow$ which coincides with the set of \myemph{zeros} (or \myemph{singular points}) of $\AFld$.

\begin{proposition}\label{prop:cond_E_for_C1_vf}
Let $\Vman\subset\Mman$ be an open subset, $\pfunc:\Vman\setminus\FixF\to\RRR$ be a $\PF$-function, $z\in\FixF\cap\Vman$, and $\{x_i\}_{i\in\NNN}\subset\Vman\setminus\FixF$ be a sequence of periodic points converging to $z$ and satisfying $(E)$.
Thus $\pfunc(x_i)\not=0$, and there exists $\MULT>0$ and a Euclidean metric on some neighbourhood of $z$ such that $d(z,x_i)<\MULT\cdot\diam(\orb_{x_i})$.
If the periods of $x_i$ are bounded above with some $\upb>0$, (in particular, condition $(A)$ holds true) then 
\begin{enumerate}
 \item[$(e_1)$~]
there exists $\eps>0$ such that $|\pfunc(x_i)|\geq \Per(x_i) >\eps$ for all $i\in\NNN$,
so the periods are bounded below as well, and
 \item[$(e_2)$~]
$j^1\AFld(z)\not=0$.
\end{enumerate}
\end{proposition}

For the proof we need some statements.
The first one is easy and we left it for the reader.

\begin{claim}\label{clm:cont_sup_gxy}
Let $X$ be a topological space, $K$ be a compact space, and $g:X\times K \to \RRR$ be a continuous function.
Then the following function $\gamma:X\to\RRR$ defined by $\gamma(x) = \sup\limits_{y\in K} g(x,y)$ is continuous.
\qed
\end{claim}

\begin{claim}\label{clm:Qxk_leq_xa}
Let $K$ be a compact manifold and $$Q=(Q_1,\ldots,Q_n):\RRR^n\times K\to\RRR^n$$ be a continuous map satisfying the following conditions:
\begin{enumerate}
  \item[\rm(a)]
$Q(\orig\times K)=\orig$.
  \item[\rm(b)]
For each $k\in K$ the map $Q_k=Q(\cdot, k):\RRR^n\to\RRR^n$ is $\Cont{1}$.
  \item[\rm(c)]
For every $i,j=1,\ldots,n$ the partial derivative $\tfrac{\partial Q_i}{\partial x_j}:\RRR^n\times K\to\RRR$ of the $i$-th coordinate function $Q_i$ of $Q$ in $x_j$ is continuous.
\end{enumerate}
In particular, conditions {\rm(b)} and {\rm(c)} hold if $K$ is a manifold and $Q$ is a $\Cont{1}$ map.
Then there exists a continuous function $\alpha:\RRR^n\to\RRR$ such that 
$$
\|Q(x,k) \| \leq \|x\|\cdot\alpha(x), \qquad (x,k)\in\RRR^n\times K.
$$

If $j^1 Q_k(\orig)=0$ for all $k\in K$, then $\alpha(\orig)=0$.
\end{claim}
\begin{proof}
It follows from (a)-(c) and the Hadamard lemma that 
$$Q_i(x,k)=\sum\limits_{j=1}^{n} x_j q_{ij}(x,k)$$
for certain continuous functions $q_{ij}:\RRR^n\times K\to\RRR$ such that $q_{ij}(0,k)=\tfrac{\partial Q_i}{\partial x_j}(0,k)$.
Hence $Q(x,k) = A(x,k) x$, where $A(x,k)=(q_{ij}(x,k))$ is an $(n\times n)$-matrix with continuous entries.
Let $S^{n-1}$ be the unit sphere in $\RRR^{n}$ centered at the origin.
Define $\alpha:\RRR^n\to\RRR$ by
\begin{equation}\label{equ:a_sup_Axk_v}
\alpha(x) = \sup_{(v,k)\in S^{n-1}\times K} \|A(x,k)\cdot v\|.
\end{equation}
Then by Claim\;\ref{clm:cont_sup_gxy} $\alpha$ is continuous.
Moreover,
$$
\|Q(x,k)\| = 
\|A(x,k) x\| =  \left\|A(x,k) \tfrac{x}{\|x\|}\right\| \cdot \|x\| \leq 
\alpha(x) \cdot \|x\|.
$$

It remains to note that if $j^1 Q_k(\orig)=0$, that is $A(\orig,k)=0$, for all $k\in K$, then by\;\eqref{equ:a_sup_Axk_v} 
$\alpha(\orig) = \sup\limits_{(v,k)\in S^{n-1}\times K} \|A(\orig,k)\cdot v\|=0$.
\end{proof}

\begin{lemma}\label{lm:estim_for_F}
Let $\AFld$ be a $\Cont{1}$ vector field in $\RRR^{n}$ such that $\AFld(\orig)=0$ and $(\AFlow_{t})$ be the local flow of $\AFld$.
Then for every $\upb>0$ there exist a neighbourhood $\Wman$ of the origin $\orig\in\RRR^n$ and a continuous function $\gamma:\Wman\to\RRR$ such that 
$$ \|\AFld(\AFlow(x,t))\| \leq \|x\| \cdot \gamma(x) $$
for all $(x,t)\in\Wman\times[-\upb,\upb]$.
If $j^1\AFld(\orig)=0$, then $\gamma(\orig)=0$.

If $\AFld$ is $\Cont{2}$, then we have a usual estimation
$ \|\AFld(\AFlow(x,t))\| \leq \MULA \cdot \|x\|^2 $
for some $\MULA>0$.
\end{lemma}
\begin{proof}
Since $\AFlow(0,t)=0$ for all $t\in\RRR$, there exists a neighbourhood $\Wman$ of $z$ such that for each $(x,t)\in\Wman\times[-\upb,\upb]$ the point $\AFlow(x,t)$ is well-defined and belongs to $\Vman$, so $\AFlow(\Wman\times[-\upb,\upb])\subset\Vman$.

Moreover, $\AFlow$ is $\Cont{1}$ and therefore it satisfies assumptions (a)-(c) of Claim\;\ref{clm:Qxk_leq_xa} with $K=[-\upb,\upb]$.
Hence there exists a continuous function $\alpha:\Wman\to\RRR$ such that 
$$
\|\AFlow(x,t)\| \leq \|x\|\cdot \alpha(x), \qquad (x,k)\in\Wman\times[-\upb,\upb].
$$

Moreover, $\AFld$ is also $\Cont{1}$ and $\AFld(\orig)=0$, whence again by Claim\;\ref{clm:Qxk_leq_xa} (for $K=\varnothing$) there exists a continuous function $\beta:\Wman\to\RRR$ such that 
$$
\|\AFld(x)\| \leq \|x\|\cdot \beta(x).
$$
Define $\gamma:\Wman\to\RRR$ by
$$
\gamma(x) = \sup_{t\in[-\upb,\upb]} \alpha(x) \cdot \beta(\AFlow(x,t)).
$$
Then by Claim\;\ref{clm:cont_sup_gxy} $\gamma$ is continuous and 
\begin{multline*}
\|\AFld(\AFlow(x,t))\|\leq 
\|\AFlow(x,t))\| \cdot \beta(\AFlow(x,t)) \leq \\ \leq
\|x\| \cdot \alpha(x) \cdot \beta(\AFlow(x,t)) \leq \|x\|\cdot \gamma(x).
\end{multline*}
Moreover, if $j^1\AFld(\orig)=0$, then $\beta(\orig)=0$.
Since in addition $\AFlow(\orig,t)=\orig$, we obtain that 
$\gamma(\orig)= \sup\limits_{t\in[-\upb,\upb]} \alpha(\orig) \cdot \beta(\orig)=0$ as well.
\end{proof}

\begin{proof}[\bf Proof of Proposition\;\ref{prop:cond_E_for_C1_vf}]
We have to show that violating either of assumptions $(e_1)$ or $(e_2)$ leads to a contradiction.

By Lemma\;\ref{lm:estim_for_F} for any $\upb>0$ there exist a neighbourhood $\Wman$ of $z$ and a continuous function $\gamma:\Wman\to\RRR$ such that 
\begin{equation}\label{equ:AAxt_dxt_gx}
\|\AFld(\AFlow(x,t))\| \leq d(x,z)\cdot \gamma(x), \qquad (x,t)\in\Wman\times[-\upb,\upb].
\end{equation}
Then
\begin{multline*}
d(z,x_i)  < \MULT\cdot\diam(\orb_{x_i}) 
      \stackrel{\eqref{equ:diam_and_len}}{\leq} 
\tfrac{\MULT}{2}\cdot l(x_i) 
      \stackrel{\eqref{equ:orb_length}~\&\;\eqref{equ:AAxt_dxt_gx}}{\leq} 
\tfrac{\MULT}{2} \cdot \Per(x_i) \cdot d(z,x_i) \cdot \gamma(x_i).
\end{multline*}
Therefore 
$$
0 < \tfrac{2}{\MULT} < \Per(x_i) \cdot \gamma(x_i) \leq |\pfunc(x_i)| \cdot \gamma(x_i).
$$
Hence if $(e_1)$ is violated, i.e.\! $\lim\limits_{i\to\infty}\Per(x_i)=0$, then $\lim\limits_{i\to\infty}\gamma(x_i)=+\infty$, which contradicts to continuity of $\gamma$ near $z$.

Suppose $j^1\AFld(z)=0$. 
Then by Lemma\;\ref{lm:estim_for_F} $\gamma(z)=0$, whence $\lim\limits_{i\to\infty}\gamma(x_i)=\gamma(z)=0$.
Therefore $\lim\limits_{i\to\infty}\Per(x_i) = \lim\limits_{i\to\infty}|\pfunc(x_i)| =  + \infty$, which contradicts to boundedness of periods of $x_i$.
\end{proof}

\section{Unboundedness of periods}\label{sect:unboundedness_of_periods}
Let $E_k$ be the unit $(k\times k)$-matrix, $C$ be a square $(k\times k)$-matrix, and $a,b\in\RRR$.
Define the following matrices:
$$
\Jord_p(C) = 
\left(
\begin{smallmatrix}
C      & 0      & \cdots & 0    \\
E_k      & C      & \cdots & 0    \\
\cdots & \cdots & \cdots & \cdots \\
0      & \cdots & E_k      & C    
\end{smallmatrix}
\right),
\ \
R(a,b) = 
\left(\begin{smallmatrix}
 a & b \\ - b & a.
\end{smallmatrix}\right),
\ \
\Jord_{p}(a\pm ib) =  \Jord_{p}(R(a,b)).
$$
For square matrices $B,C$ it is also convenient to put
$B\oplus C = \left(\begin{smallmatrix}  B & 0 \\ 0 & C \end{smallmatrix}\right)$.

\medskip 

Now let $\AFld$ be a $\Cont{1}$ vector field in $\RRR^n$ such that $\AFld(\orig)=0$.
We can regard $\AFld$ as a $\Cont{1}$ map $\AFld=(\AFld_1,\ldots,\AFld_n):\RRR^n\to\RRR^n$.
Let $$A = \left(\tfrac{\partial\AFld_{i}}{\partial x_j}(\orig)\right)_{i,j=1,\ldots,n}$$ 
be the Jacobi matrix of $\AFld$ at $\orig$.
This matrix also called the \myemph{linear part of $\AFld$ at $\orig$}.
By the real Jordan's normal form theorem $A$ is similar to the matrix of the following form:
\begin{equation}\label{equ:real_Jordan_form_for_A}
\mathop\oplus\limits_{\sigma=1}^{s} \Jord_{q_{\sigma}}(a_{\sigma} \pm i b_{\sigma}))
\ \ \oplus \ \ 
\mathop\oplus\limits_{\tau=1}^{r} \Jord_{p_{\tau}}(\lambda_\tau),
\end{equation}
where $\lambda_{\sigma} \in\RRR$ and $a_{\tau}\pm i b_{\tau} \in\CCC$ are all the eigen values of $A$.

\begin{theorem}\label{th:periods_to_infinity}
Suppose one of the following conditions holds:
\begin{enumerate}
\item
$A$ has an eigen value $\lambda$ such that $\Re(\lambda)\not=0$;
\item
The matrix\;\eqref{equ:real_Jordan_form_for_A} has either a block $\Jord_{q}(\pm ib)$ or $\Jord_{q}(0)$ with $q\geq2$;
\item
$A=0$ and there exists an open neighbourhood $\Vman$ of $\orig$ in $\RRR^n$ and a continuous $\PF$-function $\pfunc:\Vman\setminus\FixF\to\RRR$ which takes non-zero values arbitrary close to $\orig$, that is $\orig\in\cl{\Vman\setminus\pfunc^{-1}(0)}$.
\end{enumerate}
Then there exists a sequence $\{x_i\}_{i\in\NNN} \subset\Vman\setminus\FixF$ which converges to $\orig$ and such that either every $x_i$ is non-periodic, or every $x_i$ is periodic and $\lim\limits_{i\to\infty}\Per(x_i)=+\infty$.
\end{theorem}
\begin{proof}
(1) In this case by Hadamard-Perron's theorem, e.g.\;\cite{HirschPughShub:LNM538:1977}, we can find a non-periodic orbit $\orb$ of $\AFld$ such that $\orig\in\cl{\orb}\setminus\orb$.
This means that there exists a sequence $\{x_i\}_{i\in\NNN}\subset\orb$ converging to $\orig$.

\smallskip 

(2) 
Consider two cases.

(a) If \eqref{equ:real_Jordan_form_for_A} has a block $\Jord_{q}(0)$ with $q\geq2$, then it can be assumed that 
$$
A = \left(
\begin{smallmatrix}
0 & 0 &  \cdots  & 0 \\
1 & 0 &  \cdots  & 0 \\
\cdots &  \cdots &  \cdots  &  \cdots \\
\cdots &  \cdots &  \cdots  &  \cdots
\end{smallmatrix}
\right).
$$

(b) Suppose\;\eqref{equ:real_Jordan_form_for_A} has a block $\Jord_{q}(\pm ib)$ with $q\geq2$.
Then we can regard $\RRR^n$ as $\CCC^2 \oplus \RRR^{n-4}$, so the first two coordinates $x_1$ and $x_2$ are complex.
Therefore it can be supposed that 
$$
A = \left(
\begin{smallmatrix}
i b & 0 &  \cdots  & 0 \\
1 & ib &  \cdots  & 0 \\
\cdots &  \cdots &  \cdots  &  \cdots \\
\cdots &  \cdots &  \cdots  &  \cdots
\end{smallmatrix}
\right).
$$

In both cases denote by $p_1$ the projection to the first (either real or complex) coordinate, i.e. $p_1(x_1,\ldots,x_n)=x_1$.

\begin{lemma}\label{lm:nilpotent_1jet}
Let $S^{n-1}$ be the unit sphere in $\RRR^{n}$ centered at the origin $\orig$, $\eps>0$, 
$$
Y_{\eps} = \{ (x_1,\ldots,x_n) \in S^{n-1} \ : \ |x_1|\geq \eps \},
$$
and $C_{\eps}$ be the cone over $Y_{\eps}$ with vertex at $\orig$.

Then for each $\lpb>0$ there exists a neighbourhood $\Wman=\Wman_{\lpb,\eps}$ of $\orig$ such that every
$$ x \in (\Wman\cap C_{\eps})\setminus \orig$$
is either non-periodic or periodic with period $\Per(x)>\lpb$.
\end{lemma}
\begin{proof}
Let $(\AFlow_{t})$ be the local flow of $\AFld$.
Then in general, $\AFlow$ is defined only on some open neighbourhood of $\RRR^n\times0$ in $\RRR^n\times\RRR$.
Nevertheless, since $\AFlow(\orig,t)=\orig$ for all $t\in\RRR$, it follows that for each $\lpb>0$ there exists a neighbourhood $\Vman$ of $\orig$ such that $\AFlow$ is defined on $\Vman\times[-2\lpb,2\lpb]$.

We claim that in both cases there exists $c>0$ such that 
\begin{equation}\label{equ:eAtx-x}
\|e^{At}x - x \| \ \geq \ c |t\cdot x_1| \ = \ c |t\cdot p_1(x)|.
\end{equation}

(a) In this case $e^{At}x = (x_1, \,t x_1+ x_2, \,\ldots)$ and we can put $c=1$: 
$$
\|e^{At}x - x \| = \|(0, \,t x_1, \,\ldots) \| \ \geq \ |t\cdot x_1| \ = \ |t\cdot p_1(x)|.
$$

(b) Now $e^{At}x = (e^{i bt}x_1, \,e^{i bt} (t x_1+ x_2), \,\ldots)$.
Notice that we can write $e^{ibt} = 1 + t\gamma(t)$ for some smooth function $\gamma:\RRR\to\CCC\setminus\{0\}$.
Denote $c = \min\limits_{t\in[-2\lpb,2\lpb]}|\gamma(t)|$.
Then $c>0$ and 
$$
\|e^{At}x - x \| = \|(e^{i bt}x_1-x_1, \,\ldots) \| \ \geq \ |t \cdot \gamma(t) \cdot x_1| \ = \ c |t\cdot p_1(x)|.
$$

Since $\AFlow$ is a $\Cont{1}$ map and $\AFlow(\orig,t)=\orig$ for all $t\in\RRR$, it follows from Claim\;\ref{clm:Qxk_leq_xa} that there exists a continuous function $\alpha:\Vman\to[0,+\infty)$ such that 
\begin{enumerate}
 \item
$\alpha(\orig)=0$,
 \item
$\|\AFlow(x,t)-e^{At}x\|\leq \|x\|\cdot\alpha(x)$ \ for all $(x,t)\in\Vman\times[-2\lpb,2\lpb]$.
\end{enumerate}
Hence 
$$
\| \AFlow(x,t) - x \| \geq \| e^{At}x  - x \| \!-\! \| \AFlow(x,t) - e^{At}x \| \geq c |t\cdot p_1(x)| \!-\! \|x\|\alpha(x)
$$
for $(x,t)\in\RRR^n\times[-2\lpb,2\lpb]$.
Moreover, if $p_1(x)\not=0$, then
$$
\| \AFlow(x,t) - x \| \geq
c |t\cdot p_1(x)| - \|x\|\,\alpha(x) =
c |p_1(x)| \cdot \left(|t| - \tfrac{\|x\|\,\alpha(x)}{c\cdot |p_1(x)|} \right).
$$

Since $\alpha(\orig)=0$, there exists a neighbourhood $\Wman\subset\Vman$ of $\orig$ such that 
$$
\alpha(x) < c \cdot \eps \cdot \lpb, \qquad x\in\Wman.
$$

Now let $x\in(\Wman\cap C_{\eps})\setminus\orig$.
Then $|x|\leq 1$ and $|p_1(x)|\geq \eps$, whence 
$$
 \frac{\|x\|\,\alpha(x)}{c \cdot |p_1(x)|} < \frac{1 \cdot c \cdot \eps \cdot \lpb}{c \cdot \eps} = \lpb. 
$$
Therefore $\| \AFlow(x,t) - x \| \geq c \cdot \eps (|t|-\lpb)$ for $t\in[-2\lpb,2\lpb]$.
In particular, $\AFlow(x,t)\not= x$ for $t\in[\lpb,2\lpb]$.
It follows that $x$ is either non-periodic, or periodic with the period being greater than $2\lpb-\lpb=\lpb$.
\end{proof}

(3) 
Suppose that there exists a neighbourhood $\Vman$ of $\orig$ such that all points in $\Vman\setminus\FixF$ are periodic.
If the periods of points in $\Vman\setminus\pfunc^{-1}(0)$ are bounded above, i.e. condition $(A)$ holds true, then by $(e_2)$ of Proposition\;\ref{prop:cond_E_for_C1_vf} $j^1\AFld(\orig)\not=0$.
Theorem\;\ref{th:periods_to_infinity} is completed.
\end{proof}

\begin{remark}\rm
It is not true that under assumptions of Lemma\;\ref{lm:nilpotent_1jet} for \myemph{every} sequence $\{x_i\}_{i\in\NNN}$ of periodic points converging to the origin $\orig$ their periods are unbounded.
Indeed, let
$$
A = \left(
\begin{smallmatrix}
0 & 0 &   &  \\
1 & 0 &   &  \\
  &   & 0 & -1 \\
  &   & 1 & 0
\end{smallmatrix}
\right)
$$
and $\AFld(x)=Ax$ be the corresponding linear vector field on $\RRR^4$.
Then any sequence of the form $\{(0,0,x_i,y_i)\}_{i\in\NNN}$ converging to the origin consists of periodic points with periods equal to $2\pi$.
\end{remark}

On the other hand in the following special case this is so.

\begin{example}\rm
Consider the following polynomial $f(x,y)=x^{2b}+y^{2}$ on $\RRR^2$ with $2b\geq4$, and let 
$$\AFld = -f'_{y} \tfrac{\partial}{\partial x} + f'_{x}\tfrac{\partial}{\partial y}  =
-2y \tfrac{\partial}{\partial x} + 2b x^{2b-1} \tfrac{\partial}{\partial y}
$$
be the Hamiltonian vector field of $f$.
Then the orbits of $\AFld$ are the origin $\orig\in\RRR^2$ and concentric simple closed curves wrapping once around the origin, see Figure\;\ref{fig:x2y4}, and the matrix $A$ of the linear part of $\AFld$ at $z$ is nilpotent:
$A=\left(\begin{smallmatrix} 0 & -2 \\ 0 & 0 \end{smallmatrix} \right)$.
Now it follows from the structure of orbits of $\AFld$ and Lemma\;\ref{lm:nilpotent_1jet} that for any sequence $\{x_i\}\subset\RRR^2\setminus\orig$ converging to $\orig$ we have $\lim\limits_{i\to\infty}\Per(x_i)=+\infty$.
\begin{figure}[ht]
\includegraphics[height=1.5cm]{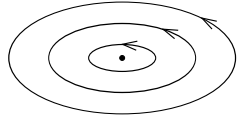}
\label{fig:x2y4}
\caption{$f(x,y)=x^{2b}+y^2$, \ ($2q\geq4$)}
\end{figure}
\end{example}

The following statement extends\;\cite[Pr.\;10]{Maks:TA:2003}, see footnote\;\ref{rem:incorrect_proof_Pr10_Maks:Shifts} on page\;\pageref{rem:incorrect_proof_Pr10_Maks:Shifts}.
\begin{proposition}\label{pr:mu_in_kerA_vanish}{\rm\cite[Pr.\;10]{Maks:TA:2003}.}
Let $\AFld$ be a $\Cont{1}$ vector field on a manifold $\Mman$, $z\in\FixF\setminus\Int{\FixF}$, $\Vman \subset \Mman$ be a neighbourhood of $z$, and $\pfunc\in\PF(\Vman)$ be a $\PF$-function.
Let also $\Wman=\cup_{\lambda\in\Lambda}\Wman_{\lambda}$ be the union of those connected components of $\Vman\setminus\FixF$ for which $z\in\cl{\Wman_{\lambda}}$.
Then each of the following conditions implies that $\pfunc=0$ on $\cl{\Wman}\cap\Vman$:
\begin{enumerate}
\item[\rm(a)]
$j^1\AFld(z)$ has an eigen value $\lambda$ such that $\Re(\lambda)\not=0$;
\item[\rm(b)]
a real normal Jordan form of $j^1\AFld(z)$ has either a block $\Jord_{q}(\pm ib)$ or $\Jord_{q}(0)$ with $q\geq2$.
\item[\rm(c)]
$j^1\AFld(z)=0$;
\item[\rm(d)]
$z\in\cl{\Int{\FixA}}\setminus\Int{\FixF}$.
\item[\rm(e)]
$\pfunc(z)=0$;
\end{enumerate}
\end{proposition}
\begin{proof}
Suppose that $\pfunc$ takes non-zero values on periodic points arbitrary close to $z$.
First we prove that every of the assumptions (a)-(d) implies (e), and then show that (e) gives rise to a contradiction.

\smallskip 

(a)$\vee$(b)$\vee$(c) $\Rightarrow$ (e).
Suppose $j^1\AFld(z)$ satisfies either of the conditions (a), (b) or (c).
Then by assumption on $\pfunc$ and Theorem\;\ref{th:periods_to_infinity} there exists a sequence $\{x_i\}_{i\in\NNN}\subset\Vman\setminus\FixF$ converging to $z$ and such that every $x_i$ is either non-periodic, or periodic but $\lim\limits_{i\to\infty}\Per(x_i)=+\infty$.

If every $x_i$ is non-periodic, then by Lemma\;\ref{lm:init_prop_ShAV} $\pfunc(x_i)=0$, whence by continuity of $\pfunc$ we obtain $\pfunc(z)=0$ as well.

Suppose every $x_i$ is periodic.
Then $\pfunc(x_i)=n_i\Per(x_i)$ for some $n_i\in\ZZZ$.
Since $\lim\limits_{i\to\infty}\Per(x_i)=+\infty$ and $\pfunc$ is continuous, it follows that $\lim\limits_{i\to\infty}n_i=0$, that is $n_i=0$ for all sufficiently large $i$, whence $\pfunc(x_i)=0$ which implies $\pfunc(z)=0$.

\smallskip 

(d) $\Rightarrow$ (c).
The assumption $z\in\cl{\Int{\FixF}}\setminus\Int{\FixF}$ means that there is a sequence $\{z_i\}_{i\in\NNN}\subset\Int{\FixF}$ converging to $z$.
But then $j^1\AFld(z_i)=0$ for all $i$.
Since $\AFld$ is $C^1$, we obtain $j^1\AFld(z)=0$ as well.

\smallskip 

(e). Suppose that $\pfunc(z)=0$.
Let $\Uman$ be a neighbourhood of $z$ with compact closure $\cl{\Uman}\subset\Vman$, and $\upb=\sup\limits_{x\in\cl{\Uman}}|\pfunc(x)|$.
Then the periods of points in $\Uman\setminus\pfunc^{-1}(0)$ are bounded above with $\upb$, that is $\pfunc$ satisfies condition $(A)$, and therefore by Lemma\;\ref{lm:prop_h_mu_A_BC_D} condition $(E)$.
Let $\{x_i\}_{i\in\NNN}\subset\Uman\setminus\FixF$ be a sequence converging to $z$ and satisfying $(E)$.
Then by Proposition\;\ref{prop:cond_E_for_C1_vf} there exists $\eps>0$ such that $|\pfunc(x_i)|>\eps$.
Since $\pfunc$ is continuous, we get $|\pfunc(z)|\geq\eps>0$, which contradicts to the assumption $\pfunc(z)=0$.
\end{proof}

\section{Proof of Theorem\;\ref{th:C1-flows-non-periodic}}
\label{sect:proof:th:C1-flows-non-periodic}
Let $\AFlow$ be the flow generated by a $\Cont{1}$ vector field $\AFld$ and $\theta\in\PPF{\Mman}$ be a non-negative generator of $\PPF{\Mman}$.
Put $Y=\theta^{-1}(0)$.
Then $Y$ is closed.

We claim that \myemph{$Y$ is also open in $\Mman$.}
Indeed, if $x$ is a non-fixed point of $\AFlow$, then by Lemma\;\ref{lm:local_uniq_sh_func} $\theta=0$ on some neighbourhood of $x$.
Suppose $x\in\FixF$.
Since $\FixF$ is nowhere dense in $\Mman$, it follows from (1) of Proposition\;\ref{pr:mu_in_kerA_vanish} that $\theta=0$ on some neighbourhood of $x$ as well.

As $\Mman$ is connected, we obtain that either $Y=\varnothing$ or $Y=\Mman$.
By Theorem\;\ref{th:description_PF_V} $\theta>0$ on $\Mman\setminus\FixF$, whence $Y\not=\Mman$ and therefore $Y=\varnothing$, so $\theta>0$ on all of $\Mman$.

Let $z\in\FixF$.
To establish\;\eqref{equ:j1Fz_ShA_periodic} it suffices to prove that
\begin{enumerate}
\item[(a)]
$j^1\AFld(z)$ has no eigen values $\lambda$ with $\Re(\lambda)\not=0$;
\item[(b)]
a real normal Jordan form of $j^1\AFld(z)$ has neither a block $\Jord_{q}(\pm ib)$ nor $\Jord_{q}(0)$ with $q\geq2$;
\item[(c)]
$j^1\AFld(z)\not=0$.
\end{enumerate}
But if either of these conditions were violated, then it would follow from Proposition\;\ref{pr:mu_in_kerA_vanish} that $\theta(z)=0$.
This contradiction completes Theorem\;\ref{th:C1-flows-non-periodic}.
\qed

\bibliographystyle{amsplain}
\bibliography{period_functions}

\def\cprime{$'$}
\providecommand{\bysame}{\leavevmode\hbox to3em{\hrulefill}\thinspace}
\providecommand{\MR}{\relax\ifhmode\unskip\space\fi MR }
\providecommand{\MRhref}[2]{%
  \href{http://www.ams.org/mathscinet-getitem?mr=#1}{#2}
}
\providecommand{\href}[2]{#2}
\begin{thebibliography}{10}

\bibitem{BoothbyWang:AnnMath:1958}
W.~M. Boothby and H.~C. Wang, \emph{On contact manifolds}, Ann. of Math. (2)
  \textbf{68} (1958), 721--734. \MR{MR0112160 (22 \#3015)}

\bibitem{Dress:Topol:1969}
Andreas Dress, \emph{Newman's theorems on transformation groups}, Topology
  \textbf{8} (1969), 203--207. \MR{MR0238353 (38 \#6629)}

\bibitem{EdwardsKennethSullivan:Top:1977}
Robert Edwards, Kenneth Millett, and Dennis Sullivan, \emph{Foliations with all
  leaves compact}, Topology \textbf{16} (1977), no.~1, 13--32. \MR{MR0438353
  (55 \#11268)}

\bibitem{Epstein:AnnMath:1972}
D.~B.~A. Epstein, \emph{Periodic flows on three-manifolds}, Ann. of Math. (2)
  \textbf{95} (1972), 66--82. \MR{MR0288785 (44 \#5981)}

\bibitem{Epstein:AIF:1976}
\bysame, \emph{Foliations with all leaves compact}, Ann. Inst. Fourier
  (Grenoble) \textbf{26} (1976), no.~1, viii, 265--282. \MR{MR0420652 (54
  \#8664)}

\bibitem{EpsteinVogt:AnnMath:1978}
D.~B.~A. Epstein and E.~Vogt, \emph{A counterexample to the periodic orbit
  conjecture in codimension {$3$}}, Ann. of Math. (2) \textbf{108} (1978),
  no.~3, 539--552. \MR{MR512432 (80c:57014)}

\bibitem{Hart:Top:1983}
David Hart, \emph{On the smoothness of generators}, Topology \textbf{22}
  (1983), no.~3, 357--363. \MR{MR710109 (85b:58105)}

\bibitem{HirschPughShub:LNM538:1977}
M.~W. Hirsch, C.~C. Pugh, and M.~Shub, \emph{Invariant manifolds}, Lecture
  Notes in Mathematics, Vol. 583, Springer-Verlag, Berlin, 1977. \MR{MR0501173
  (58 \#18595)}

\bibitem{HoffmanMann:PAMS:76}
David Hoffman and L.~N. Mann, \emph{Continuous actions of compact {L}ie groups
  on {R}iemannian manifolds}, Proc. Amer. Math. Soc. \textbf{60} (1976),
  343--348. \MR{MR0423386 (54 \#11365)}

\bibitem{Maks:TA:2003}
Sergiy Maksymenko, \emph{Smooth shifts along trajectories of flows}, Topology
  Appl. \textbf{130} (2003), no.~2, 183--204. \MR{MR1973397 (2005d:37035)}

\bibitem{Maks:reparam-sh-map}
\bysame, \emph{Reparametrization of vector fields and their shift maps}, Pr.
  Inst. Mat. Nats. Akad. Nauk Ukr. Mat. Zastos. \textbf{6} (2009), no.~2,
  489--498, arXiv:math/0907.0354.

\bibitem{MontgomerySamelsonZippin:AnnM:1956}
D.~Montgomery, H.~Samelson, and L.~Zippin, \emph{Singular points of a compact
  transformation group}, Ann. of Math. (2) \textbf{63} (1956), 1--9.
  \MR{MR0074773 (17,643g)}

\bibitem{Montgomery:AJM:1937}
Deane Montgomery, \emph{Pointwise {P}eriodic {H}omeomorphisms}, Amer. J. Math.
  \textbf{59} (1937), no.~1, 118--120. \MR{MR1507223}

\bibitem{Muller:EnsMath:1979}
Thomas M{\"u}ller, \emph{Beispiel einer periodischen instabilen holomorphen
  {S}tr\"omung}, Enseign. Math. (2) \textbf{25} (1979), no.~3-4, 309--312
  (1980). \MR{MR570315 (82b:58069)}

\bibitem{Newman:QJM:1931}
M.~H.~A. Newman, \emph{A theorem on periodic transformations of spaces}, Quart.
  J. Math., Oxford Ser. \textbf{2} (1931), 1--8.

\bibitem{Palais:MAMS:1957}
Richard~S. Palais, \emph{A global formulation of the {L}ie theory of
  transformation groups}, Mem. Amer. Math. Soc. No. \textbf{22} (1957),
  iii+123. \MR{MR0121424 (22 \#12162)}

\bibitem{Reeb:ASI:1952}
Georges Reeb, \emph{Sur certaines propri\'et\'es topologiques des vari\'et\'es
  feuillet\'ees}, Actualit\'es Sci. Ind., no. 1183, Hermann \& Cie., Paris,
  1952, Publ. Inst. Math. Univ. Strasbourg 11, pp. 5--89, 155--156.
  \MR{MR0055692 (14,1113a)}

\bibitem{Smith:AM:1941}
P.~A. Smith, \emph{Transformations of finite period. {III}. {N}ewman's
  theorem}, Ann. of Math. (2) \textbf{42} (1941), 446--458. \MR{MR0004128
  (2,324c)}

\bibitem{Sullivan:PMIHES:1976}
Dennis Sullivan, \emph{A counterexample to the periodic orbit conjecture},
  Inst. Hautes \'Etudes Sci. Publ. Math. (1976), no.~46, 5--14. \MR{MR0501022
  (58 \#18492)}

\bibitem{Sullivan:BAMS:1976}
\bysame, \emph{A new flow}, Bull. Amer. Math. Soc. \textbf{82} (1976), no.~2,
  331--332. \MR{MR0402824 (53 \#6638)}

\bibitem{Vogt:ManuscrMath:1977}
Elmar Vogt, \emph{A periodic flow with infinite {E}pstein hierarchy},
  Manuscripta Math. \textbf{22} (1977), no.~4, 403--412. \MR{MR0478175 (57
  \#17664)}

\bibitem{Wadsley:JDG:1975}
A.~W. Wadsley, \emph{Geodesic foliations by circles}, J. Differential Geometry
  \textbf{10} (1975), no.~4, 541--549. \MR{MR0400257 (53 \#4092)}

\end{thebibliography}
\end{document}